\documentclass[a4paper,12pt,twoside]{article}	
\usepackage{amssymb,amsthm,amsfonts,color,graphicx}
\usepackage{amsmath}
\usepackage{setspace}
\usepackage[margin=3.5cm]{geometry}
\newcommand{\tarc}{\mbox{\large$\frown$}}
\newcommand{\arc}[1]{\stackrel{\tarc}{#1}}
\newtheorem{theorem}{Theorem}
\DeclareMathOperator{\hr}{\mathbb H^2\times \mathbb R}
\DeclareMathOperator{\rr}{\mathbb R}

\DeclareMathOperator{\hh}{\mathbb H^2}
\DeclareMathOperator{\Ric}{{\rm Ric}}

\newtheorem{lemma}{Lemma}
\newtheorem{proposition}{Proposition}
\newtheorem{remark}{Remark}

\newtheorem{corollary}[theorem]{Corollary}
\theoremstyle{definition}\newtheorem{definition}{Definition}

\numberwithin{equation}{section}
\usepackage{verbatim}

\usepackage{layout,textcomp}
\title{The Dirichlet problem for the constant mean curvature equation in Sol$_3$}
\author{Patr\' icia Klaser\thanks{The author was partially supported by CNPq-Projeto Universal 482113/2013-8 and thanks Princeton University for the hospitality during part of the time
the research and preparation of this article were conducted.} \ and Ana Menezes}

\begin{document}
\maketitle
\begin{abstract}
We prove a version of the Jenkins-Serrin theorem for the existence of CMC graphs over bounded domains with infinite boundary data in Sol$_3$. Moreover, we construct examples of admissible domains where the results may be applied.
\end{abstract}

\begin{flushleft}
\end{flushleft}

\begin{flushleft}
\end{flushleft}

\section{Introduction}
In 1966, Jenkins and Serrin \cite{JS} studied the problem of finding necessary and sufficient conditions in a bounded domain in $\rr^2$ in order to solve the Dirichlet problem for the minimal surface equation with certain infinite boundary data. More precisely, they considered domains $\Omega\subset\rr^2$ bounded by arcs $A_1, A_2, \cdots, A_j, B_1, B_2, \cdots, B_k$ of curvature zero and arcs $C_1, C_2, ..., C_m$ of non negative curvature (normal pointing inwards), and wanted to find a function $u:\Omega\to\rr$ that satisfies the minimal surface equation and assumes $+\infty$ on the arcs $A_i$'s, $-\infty$ on $B_i$'s and a bounded prescribed data on the arcs $C_i$'s. They gave necessary and sufficient conditions for the existence of a solution on $\Omega$ in terms of the length of the boundary arcs of $\Omega$ and of inscribed polygons.

Later, Spruck \cite{S} considered the same problem for mean curvature $H>0.$ In this case it is natural to expect that the arcs $A_i$ and $B_i$ be of curvature $2H$ and $-2H,$ respectively, since a solution would be asymptotic to a vertical cylinder of mean curvature $H.$ The existence result for infinite boundary data could be thought to be proved by taking a limit of $n\to \infty$ in a sequence of solutions that take the boundary data $n$ on $A_i$ and $-n$ on $B_i.$ Nevertheless, the existence of bounded solutions is only guaranteed for convex domains. An important idea introduced in \cite{S} was to consider reflected arcs of the family $\{B_i\}$ in order to get a convex domain.

After those results many people have worked on these problems in other ambient spaces, specially in other homogeneous spaces. For instance, in $\hr$ the minimal case ($H=0$) was considered by Nelli and Rosenberg \cite{NR} (for bounded domains) and by Collin and Rosenberg \cite{CR} (for unbounded domains); and the mean curvature $H>0$ case was treated by Hauswirth, Rosenberg and Spruck \cite{HRS} (for bounded domains) and Folha and Melo \cite{FM} (for unbounded domains). For other Jenkins-Serrin type results see, for instance, \cite{AbigailCarlos, FR, GR, MRR, AnaLucia, Younes}.

Each of the ambient spaces where the problem has already been treated admits an unitary  Killing vector field that gives a vertical direction to define graphs. This is no longer true in Sol$_3.$ Recently, Nguyen \cite{Minh} considered the minimal case for both bounded and unbounded domains in Sol$_3$ and presented a warped product model of the space where vertical graphs can be considered.

Recall that Sol$_3$ can be viewed as $\rr ^3$ endowed with the Riemannian metric
\begin{equation}\label{eq-metricSol3}
ds^2=e^{2x_3}dx_1^{2} + e^{-2x_3}dx_2^2 +dx_3^2.
\end{equation} 
As described in \cite{LM}, Sol$_3$ admits exactly two foliations by totally geodesic submanifolds, the two being similar. In both foliations, each leaf is isometric to the hyperbolic plane $\mathbb{H}^2.$ The approach in \cite{Minh} was to consider one of these foliations by applying the change of coordinates
$$
x:= x_2, \ \ y:= e^{x_3}, \ \ t:= x_1,
$$
that turned Sol$_3$ into the model of a warped product, in fact, Sol$_3=\hh\times_y\rr=\{(x,y,t)\in \rr^3: y\geq0\},$ where the half-plane model is used for $\hh$ and the Riemannian metric is given by
\begin{equation}
ds^2=\frac{dx^2+dy^2}{y^2}+y^2dt^2.
\end{equation}
Seeing Sol$_3$ as a warped product gives us a natural way to define functions in $\mathbb{H}^2$ and consider their graphs in Sol$_3.$

There are two interesting properties about this model for Sol$_3$: the vertical lines are not geodesics and the mean curvature vector of a vertical plane generated by a curve $\gamma$ is related to its Euclidean geodesic curvature. In fact, if $\gamma$ is a curve in $\hh,$ then the mean curvature vector $\vec{H}_{\gamma\times\rr}$ of the vertical plane $\gamma\times\rr$ in Sol$_3$ is
$$
\vec{H}_{\gamma\times\rr}=\frac{1}{2}y^2 \vec{\kappa}_{euc},
$$
where $ \vec{\kappa}_{euc}$ denotes the Euclidean geodesic curvature vector of the curve. Hence the conditions on the boundary of a given domain can be formulated in terms of the Euclidean geodesic curvature of the arcs of the boundary.

Recall that Sol$_3$ is the Thurston geometry of smallest isometry group which has dimension three and has no positive isometry with fixed points. A consequence of the lack of isometries is that the reflected arcs idea from \cite{S} described above does not make sense in Sol$_3$ and therefore a definition of admissible domain for the case $H>0$ will be more delicate. In this paper we will work on this problem.

%The main difference in the approach by Nguyen that makes Sol$_3$ a very interesting space to work on is that in this model for Sol$_3$ the vertical Killing vector field does not have unitary norm unlike all the previous results mentioned above.
 
We briefly describe our main results. First, let us give an idea of what is an admissible domain in our context (Definition \ref{def-admsdomain}) and set some notation. A bounded domain is said to be admissible if it is bounded by $C^2$ open arcs $\{A_i\},$ $\{B_i\}$ and $\{C_i\}$ and their endpoints, with Euclidean curvature $\kappa_{euc}(A_i)=2H/y,$ $\kappa_{euc}(B_i)=-2H/y$ and $\kappa_{euc}(C_i)\ge 2H/y,$ respectively, and if the family $\{B_i\}$ is non empty, a special domain $\Omega^*$ must be well defined. We require $\Omega$ and $\Omega^*$ to be simply connected and admit a bounded subsolution to the mean curvature P.D.E.. Besides, a polygon $\mathcal{P}\subset\overline{\Omega}$ is said to be an admissible polygon if it is a curvilinear polygon whose sides are curves of Euclidean curvature $\pm 2H/y$ and vertices are chosen from among the endpoints of arcs of the families $\{A_i\}$ and $\{B_i\}.$ For an admissible polygon $\mathcal{P},$ we denote by $\alpha$ and $\beta$ the total Euclidean lengths of the arcs in the boundary $\partial \mathcal{P}$ that belong to $\{A_i\}$ and $\{B_i\},$ respectively, and by $\ell$ the Euclidean perimeter of $\cal{P}.$ We consider the quantity $$\mathcal{I}(\mathcal{P}):=\int_{\mathcal{P}}\frac{1}{y}{\rm d}a,$$ where ${\rm d}a$ is the Euclidean area element. 

Our main results about existence of solutions to the Dirichlet Problem for CMC surfaces with infinite boundary data as described before are the following.

\begin{theorem}\label{teo-Bempty}
Let $\Omega$ be an admissible domain such that the family $\{B_i\}$ is empty. Assume that %$\kappa_{euc}(C_i)\geq 2H/y$ and that 
the assigned boundary data on the arcs $\{C_i\}$ is bounded below. Then the Dirichlet Problem has a solution in $\Omega$ if and only if
\begin{equation}\label{eq-ineqalpha}
2\alpha <\ell+2H\mathcal{I}(\mathcal{P})
\end{equation}
for all admissible polygons $\mathcal{P}.$
\end{theorem}

If the family $\{A_i\}$ is empty, an analogous result holds.

\begin{theorem}\label{teo-Aempty}
Let $\Omega$ be an admissible domain such that the family $\{A_i\}$ is empty. Assume that %$\kappa_{euc}(C_i)\geq 2H/y$ and that 
the assigned boundary data on the arcs $\{C_i\}$ is bounded above. Then the Dirichlet Problem has a solution in $\Omega$ if and only if
\begin{equation}\label{eq-ineqbeta}
2\beta <\ell-2H\mathcal{I}(\mathcal{P}),
\end{equation}
for all admissible polygons $\mathcal{P}.$
\end{theorem}

The next result combines the two above for the case of the family $\{C_i\}$ being empty.

\begin{theorem}\label{teo-Cempty}
Let $\Omega$ be an admissible domain such that the family $\{C_i\}$ is empty. Then the Dirichlet Problem has a solution in $\Omega$ if and only if
\begin{equation}\label{eq-eqAB}
\alpha =\beta+2H\mathcal{I}(\Omega),
\end{equation}
and for all admissible polygons properly contained in $\overline{\Omega}$
\begin{equation}\label{eq-ineqalphaEbeta}
2\alpha <\ell+2H\mathcal{I}(\mathcal{P})
\text{ and }
2\beta <\ell-2H\mathcal{I}(\mathcal{P}).
\end{equation}
\end{theorem}

%Besides, in Section \ref{existencedomains}, we construct examples of admissible domains where the theorems above may be applied.

After studying these results some natural questions arise: Are there domains for which they apply? What can be said about them? In Section \ref{existencedomains} we construct admissible domains where Theorems \ref{teo-Bempty} and \ref{teo-Aempty} apply and in Section \ref{sec-exemplo} we exhibit a method to find domains with no $\{C_i\}$-type arcs that admit a solution to the Dirichlet Problem.

One of the main difficulties in this problem was to find the right conditions on the boundary of the domains to be considered and understand the properties of horizontal curves whose product by the vertical line has constant mean curvature.

We remark that the arguments presented here also can be adapted for proving existence of solutions of the Dirichlet problem with infinite boundary data in other warped products. %Nevertheless we work in Sol$_3$ because it is a well known homogeneous space where the problem is still open and to be in a concrete space is crucial to build examples of Scherk-type domains.

\section{Preliminaries}

In this section we will describe the model that we use for Sol$_3$ and will state some results. For more details and proofs see, for instance, Section 2 in \cite{Minh}.

The homogeneous Riemannian $3-$manifold Sol$_3$ is a Lie group which can be viewed as $\mathbb{R}^3$ with the metric \eqref{eq-metricSol3}.

A change of coordinates allows us to treat Sol$_3$ as a warped product, context in which much is known about constant mean curvature surfaces that are graphs. The diffeomorphism $\varphi: \mbox{Sol}_3\rightarrow \hh\times_f\mathbb{R}$ given by
$$\varphi(x_1, x_2, x_3)=(x_2,e^{x_3},x_1)=(x,y,t)$$
is an isometry if $\hh$ is considered with the half-plane model $\hh=\{(x,y)\in \rr^2\,|\,y>0\}$ with metric
$\dfrac{dx^2+dy^2}{y^2}$ and $f(x,y)=y,$ so that the warped metric is
\begin{equation}
ds^2=\frac{dx^2+dy^2}{y^2}+y^2dt^2.
\label{eq-met}
\end{equation}

Notice that $\|{\partial_t}\|=y$ and vertical translations are isometries. In particular, $\partial_t$ is a Killing vector field that is not unitary and vertical lines are not geodesics.

In this model, Euclidean properties of curves $\gamma$ in the half-plane are transmitted to the hypersurface $\gamma\times\mathbb{R}.$ For example: 

\begin{proposition}
Let $\gamma$ be a curve in $\hh.$ Then the mean curvature vector of $\gamma\times\rr$ in Sol$_3$ is
$$
\vec{H}_{\gamma\times\rr}=\frac{1}{2}y^2 \vec{\kappa}_{euc},
$$
where $\vec{\kappa}_{euc}$ is the Euclidean curvature vector of $\gamma.$
\label{prop_H}
\end{proposition}

Let $\Omega$ be a domain in $\hh$ and denote by $G_u$ the graph of a function $u$ over $\Omega.$ The upward unit normal vector to $G_u$ is given by
\begin{equation}
N=\frac{-y\nabla u +\frac{1}{y}\partial_t}{\sqrt{1+y^2\|\nabla u\|^2}},
\label{eq-normal}
\end{equation}
where $\nabla$ is the hyperbolic gradient operator and $\| . \|$ is the hyperbolic norm.

The graph $G_u$ of $u$ 
has constant mean curvature $H$ with respect to the normal pointing up if $u$ satisfies the equation
\begin{equation}
\mbox{div}\left(\frac{y^2\nabla u}{W}\right)=2yH,
\label{eq-cmc}
\end{equation}
where $W=\sqrt{1+y^2\|\nabla u\|^2},$ and $\mbox{div}$ denotes the hyperbolic divergence operator. If $u$ satisfies (\ref{eq-cmc}), $u$ is called a solution in $\Omega.$

\begin{definition}
Let $\Omega$ be a domain in $\hh$ and $h$ be a $C^2-$function over $\Omega.$ 
\begin{enumerate}
\item The function $h$ is a subsolution in $\Omega$ of (\ref{eq-cmc}) if
$$
\mbox{div}\left(\frac{y^2\nabla h}{W}\right)\geq2yH.
$$
\item The function $h$ is a supersolution in $\Omega$ of (\ref{eq-cmc}) if
$$
\mbox{div}\left(\frac{y^2\nabla h}{W}\right)\leq2yH.
$$
\end{enumerate}
\end{definition}

Hence the classical (bounded) Dirichlet problem in a domain $\Omega$ for the constant mean curvature equation is given by\begin{equation}\left\lbrace\begin{array}{ccl}
\mbox{div}\left(\displaystyle{\frac{y^2\nabla u}{W}}\right)&=&2yH \ \ \text{in }\ \ \Omega,\\
u&=&\varphi \ \  \text{ on }\ \  \partial \Omega.
\end{array}\right.
\label{eq-diriproblem}
\end{equation}

\begin{remark}
Throughout this paper we consider a bounded domain $\Omega\subset \hh$ that is away from the asymptotic boundary of $\hh,$ i.e., $\Omega$ is contained in some halfspace $\{(x,y)\in\hh; 0<y_0\leq y\}.$  Hence, if we need to use curves that satisfy the condition $\kappa_{euc}\geq 2H/y,$ we can take the ones that satisfy the clearer condition $\kappa_{euc}\geq 2H/y_0,$ and if we need curves with $\kappa_{euc}\leq -2H/y,$ we can take the ones that satisfy $\kappa_{euc}\leq -2H/y_0.$
\label{rem-curves}
\end{remark}

\section{Maximum principle and interior gradient estimate}\label{sec-gradestimate}

We start this section by stating a general maximum principle for sub and super solutions of the mean curvature equation for boundary data with a finite number of discontinuities (whose proof is analogous to the proof of Theorem 2.2 in \cite{HRS}).

%\begin{lemma}
%Let $v_1,v_2$ be two vectors in a finite dimensional Euclidean space. Then
%\begin{equation}
%\left\langle v_1-v_2,\frac{v_1}{W_1}-\frac{v_2}{W_2}\right\rangle=\frac{W_1+W_2}{2}\left({\left\lVert \frac{v_1}{W_1}-\frac{v_2}{W_2}\right\rVert}^2 + \left(\frac{1}{W_1}-\frac{1}%{W_2}\right)^2\right),
%\end{equation}
%where $W_i=\sqrt{1+\lVert{v_i}\rVert^2}$. In particular,
%\begin{equation}
%\left\langle v_1-v_2,\frac{v_1}{W_1}-\frac{v_2}{W_2}\right\rangle\geq {\left\lVert \frac{v_1}{W_1}-\frac{v_2}{W_2}\right\rVert}^2\geq 0.
%\end{equation}
%\label{lem-vector}
%\end{lemma}

\begin{lemma}[General Maximum Principle] 
Let $u_1$ be a subsolution and $u_2$ be a supersolution of (\ref{eq-cmc}) in a bounded domain $\Omega\subset \hh.$  Suppose that $\lim \rm{inf} \ (u_2-u_1)\geq0$ for any approach to $\partial \Omega$ with the possible exception of a finite number of points of $\partial \Omega$. Then $u_2\geq u_1$ in $\Omega$ with strict inequality unless $u_2\equiv u_1.$
\label{genmaxprin}
\end{lemma}

The next result gives an upper bound to the norm of the gradient of a solution $u$ to \eqref{eq-cmc} at a point $p$ depending on its value $u(p)$ and on the distance from $p$ to the boundary of the domain. It has two important consequences for this work: The Harnack inequality (Theorem \ref{theo-harnack}) and the Compactness Theorem (Theorem \ref{theo-compact}). It can be found in \cite{DLR}, Theorem 1. Although the statement there is weaker, their proof yields the next result as stated.

\begin{theorem}[Interior gradient estimate, \cite{DLR}]
Let $u$ be a non negative solution to \eqref{eq-cmc} on $B_R(p)\subset\hh.$ There is a constant $C=C(p,R)$ such that
$$\|\nabla u(p)\|\leq f\left(\frac{u(p)}{R}\right) \text{ for }f(t)=e^{C(t^2+1)}.$$
\label{theo-est-grad}
\end{theorem}

As a consequence of the interior gradient estimate, we have the Harnack inequality. The proof follows the same steps as in $\mathbb{R}^3,$ which was presented by Serrin in \cite{Serrin}, Theorem 5.

\begin{theorem}[The Harnack Inequality]\label{theo-harnack}
Let $u$ be a non negative solution to \eqref{eq-cmc} in $B_R(p).$ Then there is a function
$\Phi(t,r)$ such that $$u(q)\leq \Phi(m,r) \text{ and } \Phi(t,0)=t,$$ where $m=u(p)$ and $r$ is the distance from $q$ to $p.$ 

For each $t$ fixed, $\Phi(t,r)$ is a continuous strictly increasing function defined on an interval $[0,\rho(t)),$ for $\rho$ a continuous strictly decreasing function tending to zero as $t$ tends to infinity and $\lim_{r\rightarrow\rho(t)}\Phi(t,r)=+\infty.$ 
\end{theorem}

\section{Existence results}

The main theorem of this section is about the existence of solutions of \eqref{eq-diriproblem} in bounded piecewise $C^1$ domains for bounded boundary data that are continuous except in a finite subset of the boundary. This result is essential in the proof of our main results.

Given a piecewise $C^1$ domain $\Omega$, the outer curvature $\hat{\kappa}(P)$ of a point $P\in \partial \Omega$ is defined as the supremum of the curvatures of $C^2$ curves through $P$ that do not intercept $\Omega$ with normal vectors pointing to $\Omega.$ If there is not such a curve, $\hat{\kappa}(P)$ is $-\infty.$

\begin{theorem}[Existence Theorem]\label{theo_existence}
Let $\Omega\subset \mathbb{H}^2$ be a piecewise $C^1$ domain. Suppose that the outer Euclidean curvature of $\partial \Omega$ satisfies $\hat{\kappa}_{euc}(x,y)\geq 2H/y$ with possible exception in a finite set $E.$
If the equation \eqref{eq-cmc} admits a bounded subsolution in $\Omega,$
then the Dirichlet problem \eqref{eq-diriproblem} for constant mean curvature $H$ is solvable for any bounded $\varphi\in C^0(\partial D\backslash E).$ Besides, from Lemma \ref{genmaxprin}, the solution is unique.
\end{theorem}

In order to prove it, we need some preliminary results. The first is the Compactness Theorem, which follows from the gradient estimate for solutions (Theorem \ref{theo-est-grad}) and the Schauder theory for PDEs and Arzel\'a-Ascoli Theorem.

\begin{theorem}[Compactness Theorem]\label{theo-compact}
Let $\left(u_n\right)$ be a sequence of solutions of \eqref{eq-cmc} uniformly bounded in a bounded domain $\Omega.$ Then, up to a subsequence, $\left\{u_n\right\}$ converges to a solution $u$ on compact subsets of $\Omega.$ 
\end{theorem}

The other preliminary result is an application of Theorem 2 of \cite{DHL} to Sol$_3,$ which implies existence of constant mean curvature graphs taking $C^{2,\alpha}$ boundary data in $C^{2,\alpha}$ domains.

%\begin{theorem}[\cite{DHL}]\label{theo_existence_smooth}
%Let $\Omega\subset\mathbb{H}^2$  be a domain with boundary $\gamma$ of class $C^{2,\alpha}$ contained in an Euclidean disk $B_{R}(x_0,y_0+R)$ of radius $R$ centered at $(x_0,y_0+R),$ where $(x_0,y_0)\in \hh$. If the Euclidean curvature of $\gamma$ satisfies $\kappa_{euc}(x,y)\geq 2H/y$ and \begin{equation}\label{eq-desigteoexist}
%H\leq \sqrt{2}\frac{(1+2R/y_0)^{\sqrt{2}}+1}{(1+2R/y_0)^{\sqrt{2}}-1},\end{equation}
%then for any $\varphi\in C^{2,\alpha}(\gamma),$ there is a unique solution  $u\in C^{2,\alpha}(\overline{\Omega})$ of \eqref{eq-cmc}.
%\end{theorem}

\begin{theorem}[\cite{DHL}]\label{theo_existence_smooth}
Let $\Omega\subset\mathbb{H}^2$  be a domain with boundary $\gamma$ of class $C^{2,\alpha}$ contained in an Euclidean disk $B_{R}(x_0,y_0+R)$ of radius $R$ centered at $(x_0,y_0+R),$ where $(x_0,y_0)\in \hh$. If the Euclidean curvature of $\gamma$ satisfies $\kappa_{euc}(x,y)\geq 2H/y$ and \begin{equation}\label{eq-desigteoexist}
\text{either }H\leq \sqrt{2}\text{ or }
R\leq \frac{y_0}{2}\left[\left(\frac{H+\sqrt{2}}{H-\sqrt{2}}\right)^{1/\sqrt{2}}-1 \right],\end{equation}
then for any $\varphi\in C^{2,\alpha}(\gamma),$ there is a unique solution  $u\in C^{2,\alpha}(\overline{\Omega})$ of \eqref{eq-cmc}.
\end{theorem}

To see that the above result is a consequence of Theorem 2 of \cite{DHL}, notice that in Sol$_3,$ $\Ric_{\rm Sol_3}\geq -2,$ the mean curvature of the cylinder $\gamma \times \mathbb{R}$ is $y\kappa_{euc}/2$ (see Proposition \ref{prop_H}) and that $\Omega$ is contained in an Euclidean disk $B_{R}(x_0,y_0+R)$ if and only if it is contained in the hyperbolic disk of center $(x_0,\sqrt{y_0^2+2Ry_0})$ and radius $ \ln(1+2R/y_0)/2.$ Therefore, Theorem 2 of \cite{DHL} applies if $\ln(1+2R/y_0)/2\leq \coth^{-1}(H/\sqrt{2})/\sqrt{2},$ which is equivalent to \eqref{eq-desigteoexist}.

In \cite{DLR} an interior gradient estimate that implies Theorem \ref{theo_existence_smooth} for only continuous boundary data is obtained. The proof of how the interior gradient estimate implies the generalization to $C^0$ boundary data is made in Section 3 of \cite{DLR}: The idea is to take sequences of smooth boundary data that approximate the continuous data $\varphi$ from above and from below and use the interior gradient estimates to guarantee the convergence in compact subsets of the domain of sequence of solutions.

%The proof of Theorem \ref{theo_existence} follows the same steps as the proof of Theorem 3.14 of \cite{Minh}, using first the Perron method to obtain a solution and then Theorem \ref{theo_existence_smooth} to build local barriers and prove that the solution extends continuously to $\partial \Omega\backslash E.$ 

\begin{proof}[Proof of Theorem \ref{theo_existence}]
We apply the Perron method as in \cite{Serrin} to obtain a solution. 
Since \eqref{eq-cmc} admits a bounded subsolution in $\Omega,$ by translating it downwards if necessary, the set $$S_\varphi=\{v\in C^2(\Omega)\cap C^0(\overline{\Omega})\,|\,v \text{ is a subsolution to \eqref{eq-cmc} in } \Omega \text{ and } v\le \varphi \text{ on } \partial \Omega \}$$ is non empty. If $M=\sup\varphi,$ then the constant function $v=M$ is above any function in $S_\varphi.$  Let $m=\inf_{\overline{\Omega}} v_0,$ for some $v_0\in S_\varphi.$ Using the Compactness Theorem, the Perron method implies that $u=\sup_{v\in S_\varphi} v$ is a solution to \eqref{eq-cmc} in $\Omega$ with $m\le u \le M.$

It remains to show that $u$ extends continuously to $\partial \Omega\backslash E.$ This is a consequence of the existence of local barriers given by Theorem \ref{theo_existence_smooth}. More precisely, for any $q\in \partial \Omega\backslash E$ and for any $\varepsilon >0$ sufficiently small, consider $U_\varepsilon$ the subset of $\Omega$ obtained  by smoothing the boundary of $\Omega\cap B_\varepsilon(q),$ for $B_\varepsilon(q)$ the Euclidean ball centered at $q.$ We define $w^+$ and $w^-$ the solutions to \eqref{eq-cmc} in $U_\varepsilon$ with boundary data
$$w^+=\varphi^+ \text{ on }\partial U_\varepsilon \text{ and }w^-=\varphi^- \text{ on }\partial U_\varepsilon,$$
for $\varphi^+, \varphi^- \in C^0(\partial U_\varepsilon),$ such that $\varphi^+(q)=\varphi^-(q)=\varphi(q),$ $\varphi^-\leq \varphi\leq\varphi^+$ on $\partial U_\varepsilon\cap \partial \Omega,$ $\varphi^+=M$ on $\partial U_\varepsilon \cap \Omega$ and $\varphi^-=m$ on $\partial U_\varepsilon \cap \Omega.$ Then $w^-$ and $w^+$ are lower and upper barriers,
$w^-\leq u\leq w^+$ in $U_\varepsilon,$ which implies the continuity of $u$ up to the boundary.
\end{proof}

A particular case where a subsolution exists occurs when the domain $\Omega$ is contained in a disk as in Theorem \ref{theo_existence_smooth}. We remark that for any $H\ge 0,$ if a domain is taken sufficiently small, the Dirichlet problem is solvable, as we can see as a consequence of the next corollary.

\begin{corollary}\label{cor-exist-particular}
Let $\Omega\subset \mathbb{H}^2$ be a piecewise $C^1$ domain. Suppose that $\overline{\Omega}$ is contained in an Euclidean disk of radius $R$ centered at $(x_0,y_0+R)$, where $(x_0,y_0)\in \hh$, and that the outer Euclidean curvature of $\partial \Omega$ satisfies $\hat{\kappa}_{euc}(x,y)\geq 2H/y$ with possible exception in a finite set $E.$
If \eqref{eq-desigteoexist} holds,
%\leq \sqrt{2}\frac{(1+2R/y_0)^{\sqrt{2}}+1}{(1+2R/y_0)^{\sqrt{2}}-1},$$
then the Dirichlet problem for constant mean curvature $H$ is solvable for any bounded $\varphi\in C^0(\partial \Omega\backslash E).$
\end{corollary}

\begin{proof}
Notice that given a bounded function $\varphi\in C^0(\partial \Omega \backslash E),$ 
we can find a solution to \eqref{eq-cmc} defined on $\overline{\Omega}$ that is below $\varphi$ on $\partial \Omega.$ For that, take a larger $C^{2,\alpha}$ domain $\widetilde{\Omega}$ containing $\overline{\Omega},$ such that $\widetilde{\Omega}$ is contained in the disk of radius $R$ and, by Theorem \ref{theo_existence_smooth} applied to the Dirichlet problem on $\widetilde{\Omega}$ with boundary data $u=\inf \varphi,$ there is a solution $w,$ which is (or can be translated downwards to be) below $\varphi$ on $\partial \Omega.$ Hence, by Theorem \ref{theo_existence}, the Dirichlet problem has a solution in $\Omega.$
\end{proof}

\section{Some local barriers}
\begin{lemma}
Let $\Omega$ be a domain and $\gamma$ be a $C^2$ arc of $\partial\Omega.$
\begin{enumerate}
\item[(i)] If $\kappa_{euc}(x,y)< 2H/y,$ then any point in the interior of $\gamma$ admits a neighborhood $U\subset \Omega$ in which there is a supersolution $u^+$ of \eqref{eq-cmc} with exterior normal derivative $\frac{\partial u^+}{\partial \eta}=+\infty$ along $\gamma.$

\item[(ii)] If $\gamma$ has $\kappa_{euc}(x,y)< -2H/y.$ Then any point in the interior of $\gamma$ admits a neighborhood $U\subset \Omega$ in which there is a subsolution $u^-$ of \eqref{eq-cmc} with exterior normal derivative $\frac{\partial u^-}{\partial \eta}=-\infty$ along $\gamma.$
\end{enumerate}

\label{lem-BarriersInfiniteDerivative}
\end{lemma}

\begin{proof}
Let us prove the first item. Consider $w$ a function of the hyperbolic distance $r$ to $\gamma.$ 
By definition, the function $w$ is a supersolution to \eqref{eq-cmc} with normal derivative $\frac{\partial w}{\partial \eta}=+\infty$ along $\gamma$ if and only if
\begin{equation} 
\mbox{div}\left(\frac{y^2\nabla w}{W}\right)=
\frac{1}{W^3}\left[(2yw'+y^3w'^3)\langle \nabla y,\nabla r\rangle+y^2 w''\right]+\frac{y^2 w'\Delta r}{W} \leq 2yH,
\label{eq-supersol}
\end{equation}
where $W=\sqrt{1+y^2w'^2}$ and the gradient and the Laplacian are taken in the hyperbolic metric and also $\lim_{r\rightarrow 0}w'(r)=-\infty.$

Since we only have assumptions on the Euclidean curvature of the boundary, let us relate the hyperbolic Laplacian with the Euclidean one. Denoting by $\gamma_r$ the curve in $\Omega$ parallel to $\gamma$ with hyperbolic distance $r$ apart from $\gamma,$ we have for $r>0$ that 
\begin{equation}\label{eq-relacaok_euc-k_hip}
-\Delta r=\kappa(\gamma_r)=y\kappa_{euc}(\gamma_r)+\frac{1}{y}\langle \nabla y,\nabla r\rangle.
\end{equation}
We define $$\widetilde{W}=\frac{W}{-yw'}=\sqrt{1+\left(y^2w'^2\right)^{-1}}$$
and we can rewrite \eqref{eq-supersol} with the Euclidean curvature

\begin{equation}\label{eq-desig-supersol}
\frac{1}{\widetilde{W}^3}\left[-\frac{w''}{yw'^3}\right]+\frac{y^2\kappa_{euc}(\gamma_r)}{\widetilde{W}}-\frac{\langle\nabla y,\nabla r\rangle}{y^2w'^2\widetilde{W}^3}\leq 2yH.
\end{equation}

Take $w(r)=-r^a,$ for $a\in (0,1/2).$ Hence when $r\to0,$ we get 
$$w'(r)\to-\infty,\, \widetilde{W}^3\to1,\, \frac{\langle\nabla y,\nabla r\rangle}{y^2w'^2\widetilde{W}^3}\to0\text{ and }\frac{w''}{w'^3}\to0 \text{ (since } a\in(0,1/2)).$$ 
Then given $\varepsilon>0,$ the first and third terms of \eqref{eq-desig-supersol} can be assumed to have absolute value less than $\varepsilon/3,$ if $r$ is sufficiently small.
Moreover, for sufficiently small $r$ and restricting to a neigboorhood $U$ where $y$ does not vary much, from relation \eqref{eq-relacaok_euc-k_hip}, the Euclidean curvature of the parallel curves $\gamma_r$ also remains bounded close to $2H/y,$ implying \eqref{eq-desig-supersol} in $U$
%$$
%\frac{1}{\widetilde{W}^3}\left[-\frac{w''}{yw'^3}\right]+\frac{y^2\kappa_{euc}(\gamma_r)}{\widetilde{W}}-\frac{\langle\nabla y,\nabla r\rangle}{y^2w'^2\widetilde{W}^3}\leq 2yH
%$$
and then $u^+=w$ is a supersolution in $U.$

To prove the second item, take $u^-=-w$ so $u^-$ is a subsolution in a neighborhood $U.$
\end{proof}

The following lemma is a variant of the maximum principle for sub- or supersolution with infinite boundary derivative. This result together with the barriers constructed in Lemma \ref{lem-BarriersInfiniteDerivative} allows us to obtain a bound for solutions in neighborhoods of the boundary (Lemma \ref{lem-viz}).

\begin{lemma}\label{lem-pdomaxcominfiniteslope}
Let $\Omega$ be a domain bounded by the union of two closed arcs $\gamma_1$ and $\gamma_2,$ where $\gamma_2$ is of class $C^1.$ Let $u\in C^2(\Omega)\cap C^1(\gamma_2)$ and $v\in C^2(\Omega)\cap C^0(\overline{\Omega})$ be respectively a solution and a subsolution of \eqref{eq-cmc} in $\Omega$ and assume that $\frac{\partial v}{\partial \eta}=-\infty$ along $\gamma_2.$ If $\liminf(u-v)\geq 0$ for any approach to a point of $\gamma_1,$ then $v\leq u$ in $\Omega.$
\end{lemma}

The proof follows from the General Maximum Principle (Lemma \ref{genmaxprin}) exactly as stated in \cite{S}.

\begin{lemma}\label{lem-localbarr}
Let $u$ be a solution of $(\ref{eq-cmc})$ in a domain $\Omega,$ $\gamma\subset \partial\Omega$ be a $C^2$ arc, and suppose that $m\leq u\leq M$ on $\gamma$ for some constants $m,M.$ Then there exists a constant $c=c(\Omega)$ (only depending on $\Omega$) such that for any compact $C^2$ subarc $ \gamma' \subset \gamma,$

\begin{enumerate}
\item[(i)] if $\kappa_{euc}(\gamma')\geq 2H/y$ with strict inequality except for isolated points, then there is a neighborhood $U$ of $\gamma'$ in $\overline\Omega$ such that $u\geq m-c$ in $U;$

\item[(ii)] if $\kappa_{euc}(\gamma')> -2H/y,$ then there is a neighborhood $U$ of $\gamma'$ in $\overline\Omega$ such that $u\leq M+c$ in $U.$
\end{enumerate}
\label{lem-viz}
\end{lemma}

\begin{proof}
If $\kappa_{euc}(\gamma')\geq 2H/y$ with strict inequality except for isolated points, then, assuming that $\gamma'$ is small enough, there is a curve $\delta$ with $\kappa_{euc}(\delta)> 2H/y$ such that $\gamma'\cup \delta$ bounds a domain $U.$
Reverting the orientation of $\delta$ so that it encloses $U,$ it has $\kappa_{euc}(\delta)<- 2H/y.$

By approximating $\delta$ to $\gamma'$ if necessary, we can assume that $U$ is contained in the neighborhood of $\delta$ given by Lemma \ref{lem-BarriersInfiniteDerivative}, where $u^-$  is defined.
Applying Lemma \ref{lem-pdomaxcominfiniteslope} for the subsolution $v=u^-+m-\sup_{U}u^-,$ we conclude that $$u\geq m-(\sup_{U}u^- - \inf_{U}u^-)$$  in $U,$ proving the first part of the lemma for $c=\displaystyle{\sup_{U}u^- - \inf_{U}u^-.}$

We remark that the assumption on the size of $\gamma'$ is not important since we can divide it into small pieces and obtain $U$ as the union of a finite number of neighborhoods.

An analogous argument with the supersolution $u^+$ (by taking $v=u^{+}+M-\inf_{U}u^+$) implies the second assertion of the lemma.
\end{proof}

\section{Flux formula}\label{sec-flux}
In this section we will describe some flux formulas that we will need in order to establish our existence theory for solutions with infinite boundary values. The definition of the flux of a function presented here is very similar to the ones presented in previous results in other ambient spaces. The only essential difference is that  we deal with Killing graphs whose Killing vector field is not unitary ($|\partial_t|=y$), hence the norm of this vector field naturally appears in the definition of the flux.

Let $u\in C^2(\Omega)\cap C^1(\bar{\Omega})$ be a solution of (\ref{eq-cmc}) in a domain $\Omega\subset \hh.$ Then integrating (\ref{eq-cmc}) over $\Omega$ gives
\begin{equation}
\int_{\Omega}2yH{\rm d}\sigma=\int_{\partial \Omega}\langle yX_u, \nu\rangle {\rm d}s,
\label{eq-flux1}
\end{equation}
where ${\rm d}\sigma$ is the area element in $\mathbb{H}^2$, $X_u=\frac{y\nabla u}{\sqrt{1+y^2|\nabla u|^2}}$ and $\nu$ is the outer normal to $\partial \Omega.$ The right-hand integral is called the flux of $u$ across $\partial \Omega.$ Motivated by this equality, we define the flux of $u$ across any subarc of $\partial\Omega$ as follows.

\begin{definition}
Let $\gamma$ be a subarc of $\partial\Omega.$ Take $\eta$ a simple smooth curve in $\Omega$ so that $\gamma\cup\eta$ bounds a simply connected domain $\Delta_\eta.$ We define the flux of $u$ across $\gamma$ to be
\begin{equation}
F_u(\gamma)=2H\int_{\Delta_\eta}y {\rm d}\sigma -\int_{\eta}\langle yX_u, \nu\rangle {\rm d}s.
\end{equation}
\end{definition}

Observe that the first integral does not depend on $u$, it only depends on the domain. Moreover, the definition does not depend on the choice of $\eta.$ In fact, let $\widetilde{\eta}$ be another choice of curve and consider the 2-chain $\mathcal C$ with oriented boundary $\eta-\tilde{\eta}.$ Using (\ref{eq-cmc}) and the divergence theorem on $\mathcal C$ we get
$$
2H\int_{\Delta_{\tilde\eta}}y {\rm d}\sigma -2H\int_{\Delta_\eta}y {\rm d}\sigma= \int_{\tilde\eta}\langle yX_u, \nu\rangle {\rm d}s-\int_{\eta}\langle yX_u, \nu\rangle {\rm d}s.
$$
Therefore, the definition is well posed. Notice that if $u\in C^1(\Omega\cup\gamma),$ then we can choose $\eta$ to be $\gamma$ with the inverse orientation, and then $F_u(\gamma)=\int_{\gamma}\langle yX_u, \nu\rangle {\rm d}s.$

Notice that the definition of flux makes sense for any curve $\gamma$ contained in $\Omega.$ In fact, if $\gamma\subset \Omega$ we can consider a subdomain $U\subset\Omega$ such that $\gamma\subset\partial U$ and use the definition above.

The proof of the next three lemmas follows the same steps as the proof of Proposition 4.6 in \cite{Minh}.

\begin{lemma}
Let $u$ be a solution of (\ref{eq-cmc}) in a domain $\Omega.$ Then
\begin{enumerate}
\item $\int_{\Omega}2yH{\rm d}\sigma=\int_{\partial \Omega}\langle yX_u, \nu\rangle {\rm d}s;$

\item For every curve $\gamma$ in $\Omega$ with $\ell_{euc}(\gamma)<\infty,$ we have $|F_u(\gamma)|< \ell_{euc}(\gamma);$

\item For every curve $\gamma$ in $\overline{\Omega}$ with $\ell_{euc}(\gamma)<\infty,$ we have $|F_u(\gamma)|\leq \ell_{euc}(\gamma).$
\end{enumerate}
\label{lem-flux1}
\end{lemma}
%\begin{proof}
%The first item is immediate. Let us prove the other two. 

%First suppose that $\gamma\subset \Omega.$ Then
%$$
%|F_u(\gamma)|\leq \int_{\gamma}|\langle yX_u,\nu\rangle|{\rm d}s< \int_{\gamma}y{\rm d}s=\ell_{euc}(\gamma),
%$$ 
%where we use that $|X_u|<1$ in $\Omega.$

%Now suppose that $\gamma$ is not entirely contained in $\Omega.$ Given any $\epsilon>0,$ we can choose a curve $\eta\subset\Omega$ such that $\ell_{euc}(\eta)\leq \ell_{euc}(\gamma) + \epsilon.$ Hence 
%$$
%|F_u(\gamma)|=|F_u(\eta)|\leq \ell_{euc}(\eta)\leq \ell_{euc}(\gamma)+\epsilon.
%$$
%Since $\epsilon$ is arbitrary, this proves the result.
%\end{proof} 

\begin{lemma}
Let $u$ be a solution of (\ref{eq-cmc}) in a domain $\Omega$ and $\gamma\subset \partial\Omega$ be a piecewise $C^2$ arc satisfying $\kappa_{euc}(\gamma)\geq 2H/y$ and so that $u$ is continuous on $\gamma.$ Then
\begin{equation}
\left|\int_{\gamma}\langle yX_u, \nu \rangle{\rm d}s \right|<\ell_{euc}(\gamma).
\label{eq-flux2}
\end{equation}
\label{lem-flux2}
\end{lemma}
%\begin{proof}
%It suffices to prove (\ref{eq-flux2}) for a small subarc $\xi$ of $\gamma.$ Take a point $p\in\xi$ and consider a small neighborhood $V_\epsilon\subset\Omega$  such that $\xi\subset\partial V_\epsilon$ and $\partial V_\epsilon$ satisfies $\kappa_{euc}\geq 2H/y$ (see Remark \ref{rem-curves}). Then by the existence theorem (Corollary \ref{cor-exist-particular}), there is a solution $v$ of (\ref{eq-cmc}) in $V_\epsilon$ with $v=u+1$ on $\xi$ and $v=u$ on $\partial V_\epsilon\setminus \xi$. 
%Using Lemma \ref{lem-vector} and the fact that $u$ and $v$ satisfy the constant mean curvature equation \eqref{eq-cmc}, we obtain

%$$\begin{array}{ccl}
%0&  < &\displaystyle{ \int_{V_{\xi}}\langle\nabla v-\nabla u,yX_v-yX_u\rangle{\rm d}\sigma}\\
%&&\\
%&=&\displaystyle{\int_{V_\xi}\mbox{div} ((v-u)(yX_v-yX_u)){\rm d}\sigma}\\
%&&\\
%&=&\displaystyle{\int_{\partial V_\xi} \langle (v-u)(yX_v-yX_u), \nu\rangle {\rm d}s}\\
%&&\\
%&=& \displaystyle{\int_{\xi}\langle yX_v-yX_u,\nu\rangle {\rm d}s}\\
%&&\\
%&=& \displaystyle{F_{v}(\xi)-F_u(\xi).}
%\end{array}$$

%Then 
%$$
%F_u(\xi)<F_v(\xi)\leq \ell_{euc}(\xi).
%$$
%Together with the second part of Lemma \ref{lem-flux1} we conclude
%$$
%\left|\int_{\xi}\langle yX_u, \nu \rangle\right|<\ell_{euc}(\xi).
%$$
%\end{proof}

\begin{lemma}
Let $u$ be a solution of (\ref{eq-cmc}) in a domain $\Omega$ and $\gamma\subset\partial \Omega$ be a piecewise $C^2$ arc.
\begin{enumerate}
\item If $u$ tends to $+\infty$ on $\gamma,$ then $\kappa_{euc}(\gamma)=2H/y$ and 
$$
\int_{\gamma}\langle yX_u, \nu\rangle {\rm d}s = \ell_{euc}(\gamma).
$$
\item If $u$ tends to $-\infty$ on $\gamma,$ then $\kappa_{euc}(\gamma)=-2H/y$ and 
$$
\int_{\gamma}\langle yX_u, \nu\rangle {\rm d}s = -\ell_{euc}(\gamma).
$$
\end{enumerate}
\label{lem-flux3}
\end{lemma}

The following lemma is a simple extension of Lemma \ref{lem-flux3}.

\begin{lemma}\label{lem-flux4}
Let $\Omega$ be a domain and $\gamma\subset \partial\Omega$ be a compact piecewise $C^2$ arc. Let $\{u_n\}$ be a sequence of solutions of $(\ref{eq-cmc})$ in $\Omega$ such that each $u_n$ is continuous on $\gamma.$

\begin{enumerate}
\item If the sequence diverges to $+\infty$ uniformly on compact subsets of $\gamma$ while remaining uniformly bounded on compact subsets of $\Omega,$ then
$$
\lim_{n\rightarrow +\infty} \int_{\gamma}\langle yX_{u_n}, \nu\rangle {\rm d}s =\ell_{euc}(\gamma).
$$
\item If the sequence diverges to $-\infty$ uniformly on compact subsets of $\gamma$ while remaining uniformly bounded on compact subsets of $\Omega,$ then
$$
\lim_{n\rightarrow +\infty} \int_{\gamma}\langle yX_{u_n}, \nu\rangle {\rm d}s =-\ell_{euc}(\gamma).
$$
\end{enumerate}
\end{lemma}

We have one more useful property of the flux given by the next lemma.

\begin{lemma}
Let $\Omega$ be a domain and $\gamma\subset \partial\Omega$ be a compact piecewise $C^2$ arc with $\kappa_{euc}(\gamma)=2H/y$. Let $\{u_n\}$ be a sequence of solutions of $(\ref{eq-cmc})$ in $\Omega$ such that each $u_n$ is continuous on $\gamma.$ Then if the sequence diverges to $-\infty$ uniformly on compact sets of $\Omega$ and remains uniformly bounded on compact subsets of $\gamma,$ we get
$$
\lim_{n\rightarrow +\infty} \int_{\gamma}\langle yX_{u_n}, \nu\rangle {\rm d}s =\ell_{euc}(\gamma).
$$
\label{lem-flux5}
\end{lemma}

\begin{proof}
The idea to prove this result is basically the same idea as the proof of Lemma \ref{lem-flux3}. 

Denote by $G_{u_n}$ the graph of $u_n.$ Since each $G_{u_n}$ is a stable constant mean curvature surface, there exists $\delta>0$ (depending only on $\Omega$) such that, if $p_k$ is sufficiently far from the boundary (a condition that depends only on $\delta$), a neighborhood of each point $(p_k,u_n(p_k))$ in $G_{u_n}$ is a graph of bounded geometry over a disk of radius $\delta$ centered at the origin of the tangent plane $T_{p_k}G_{u_n}.$

 Now let $p\in\gamma$ and consider a sequence of points $p_k\in\Omega$ that converges to $p.$ Since $u_n(p_k)\to-\infty,$ then, for each $k,$ there exists $n_k$ large such that for $n\geq n_k,$ a neighborhood of $(p_k,u_{n}(p_k))$ is a graph of bounded geometry. Now consider the sequence of graphs, denoted by $G_{p_k}(\delta),$ at each point $(p_k, u_{n_k}(p_k))$ that have bounded geometry. If we denote by $N_{p_k}$ the unit normal vector to $G_{p_k}(\delta)$ at $p_k,$ we can prove exactly as in Lemma \ref{lem-flux3} that $N_{p_k}$ converges to a horizontal vector. Since $u_{n_k}\to-\infty$ on $p_k$ and $\{u_{n_k}(p)\}$ is bounded, we know that $\langle \nabla u_{n_k}, \nu\rangle \geq 0$ and then $\langle N_{u_{n_k}}, \nu\rangle\leq 0.$ Hence we conclude that $N_{u_{n_k}}(p)=-\nu.$ Since $p$ is arbitrary, we get
 $$
\lim_{n\rightarrow +\infty} \int_{\gamma}\langle yX_{u_n}, \nu\rangle {\rm d}s = \lim_{n\rightarrow +\infty} \int_{\gamma}y\langle -N_{u_n}, \nu\rangle {\rm d}s= \int_{\gamma}y {\rm d}s=\ell_{euc}(\gamma).
$$
\end{proof}

\section{Monotone convergence theorem and the divergence set}\label{sec-monotone}

As in the Euclidean case, we have the following consequence of Lemma \ref{lem-flux3} and the Local Harnack Inequality (Theorem \ref{theo-harnack}).

\begin{theorem}
Let $\{u_n\}$ be a monotone increasing or decreasing sequence of solutions of $(\ref{eq-cmc})$ on a  domain $\Omega\subset \hh.$ If the sequence is bounded at some point of $\Omega,$ then there exists a nonempty open set $U\subset \Omega$ such that the sequence $\{u_n\}$ converges to a solution of $(\ref{eq-cmc})$ in $U.$ The convergence is uniform on compact subsets of $U$ and the divergence is uniform on compact subsets of $V=\Omega\setminus U.$ If $V$ is nonempty, then its boundary $\partial V$ consists of arcs of curvature $\kappa_{euc}=\pm 2H/y$ and arcs of $\partial \Omega.$ These arcs are convex to $U$ for increasing sequences and concave to $U$ for decreasing sequences. In particular, no component of $V$ can consist of a single interior arc.
\label{theo-monotone}
\end{theorem}

Moreover, we have the following property for the divergence set $V.$

\begin{lemma}\label{lem-Vboundary}
Let $\Omega$ be a domain bounded in part by an arc $\gamma$ with $\kappa_{euc}(\gamma)\geq 2H/y.$ Let $\{u_n\}$ be a monotone increasing or decreasing sequence of solutions of $(\ref{eq-cmc})$ in $\Omega$ with each $u_n$ continuous in $\Omega\cup\gamma.$ Suppose $\alpha$ is an interior arc of $\Omega$ with curvature $\kappa_{euc}(\alpha)=2H/y$ forming part of the boundary of the divergence set $V.$ Then $\alpha$ cannot terminate at an interior point of $\gamma$ if $\{u_n\}$ either diverges to $\pm \infty$ on $\gamma$  or remains uniformly bounded on compact subsets of $\gamma.$
\end{lemma}

\begin{proof}

Let $\alpha\subset\partial V$ be an interior arc of $\Omega$ with $\kappa_{euc}(\alpha)=2H/y$ and suppose that it terminates at an interior point $p\in\gamma.$ Up to restricting ourselves to a subarc of $\gamma$ that contains $p,$ we can assume that $\gamma$ is a $C^2$ arc. By Lemma \ref{lem-flux3}, the sequence $\{u_n\}$ cannot diverge to $-\infty$ on $\gamma;$ and if the Euclidean curvature of $\gamma$ is not identically $2H/y,$ then the sequence $\{u_n\}$ cannot diverge to $+\infty$ on $\gamma.$ Hence if the Euclidean curvature of $\gamma$ is not identically $2H/y,$ the sequence $\{u_n\}$ remains uniformly bounded on compact subsets of $\gamma,$ but it follows from Lemma \ref{lem-viz} that a neighborhood of $\gamma$ is contained in $U,$ a contradiction. Thus we can assume that $\kappa_{euc}(\gamma)=2H/y.$

Suppose $\{u_n\}$ diverges to $+\infty$ on $\gamma$ and there exists only one such $\alpha$ that terminates at $p.$  It follows that $\{u_n\}$ diverges to $+\infty$ on $\alpha.$ In fact, since there is only one such $\alpha,$ then one sub arc $\gamma'$ of $\gamma$ with $p$ as an endpoint is contained necessarily  in $\partial V,$ hence we can choose a point $q\in\alpha$ and $r\in\gamma'$ such that the triangle with vertices $p,q,r$ is entirely contained in $V;$ thus since the divergence in $V$ is uniform on compact subsets and we already know that $u_n\to+\infty$ on $\gamma,$ we conclude that $\{u_n\}$ diverges to $+\infty$ on $\alpha.$

Now let $q$ be a point of $\alpha$ close to $p$ and choose a point $r$ on $\gamma$ near $p$ so that the Euclidean geodesic segment $\overline{rq}$ lies in $U.$ Let $T$ be the triangle formed by $\overline{rq}$ and the arcs $\stackrel \frown {qp} \ \subset \alpha$ and $ \stackrel \frown {pr} \ \subset\gamma$ of curvature $2H/y.$

By (\ref{eq-flux1}) we have
$$
\int_T 2yH{\rm d}\sigma=F_{u_n}(\stackrel \frown {qp})+F_{u_n}(\stackrel \frown {pr}) +F_{u_n}(\overline {rq}).
$$

Since by Lemma \ref{lem-flux4} we know that
$$
\lim_{n\to\infty}F_{u_n}(\stackrel \frown {qp})=\ell_{euc}(\stackrel \frown {qp}), \ \ \lim_{n\to\infty}F_{u_n}(\stackrel \frown {pr})=\ell_{euc}(\stackrel \frown {pr})
$$
and by Lemma \ref{lem-flux1} (item $2$), $F_{u_n}(\overline {rq})>-\ell_{euc}(\overline {rq}),$ then we get
\begin{equation}
\frac{\int_T 2yH{\rm d}\sigma}{\ell_{euc}(\overline {rq})}\geq \frac{\ell_{euc}(\stackrel \frown {qp})+\ell_{euc}(\stackrel \frown {pr})}{\ell_{euc}(\overline{rq})}-1.
\label{eq-1}
\end{equation}

Keeping $p$ fixed, we move the point $q$ to $\bar{q}$ and $r$ to $\bar{r}$ along the same arcs so that $\ell_{euc}(\stackrel \frown {\bar{q}p})=\lambda \ell_{euc}(\stackrel \frown {qp})$ and $\ell_{euc}(\stackrel \frown {p\bar{r}})=\lambda \ell_{euc}(\stackrel \frown {pr})$ for $\lambda<1.$ Observe that $\int_T 2yH{\rm d}\sigma$ has a quadratic dependency on $\lambda$ while the other terms have a linear dependency. Then the left-hand side of (\ref{eq-1}) tends to zero when $\lambda\to0$, and the right-hand side remains uniformly positive, a contradiction.  So if such $\alpha\subset \partial V$ exists, it should not be unique. However, if we assume that there are at least two arcs $\alpha_1,\alpha_2\subset\partial V$ that terminate at $p,$ then again we could find a triangle $T\subset U$ whose edges are two constant curvature $2H/y$ arcs (where $u_n\to+\infty$) and an Euclidean geodesic segment as before (perhaps with $\partial T\cap \gamma=\{p\}$), and the same argument would give us a contradiction.

If the sequence remains uniformly bounded on compact subsets of $\gamma,$ we choose $r$ on $\gamma$ so that $T$ is contained in $V.$ By Lemma \ref{lem-localbarr} the sequence must diverge to $-\infty$ in $V.$ We now reach a contradiction as above using Lemma \ref{lem-flux5}.

\end{proof}

\section{Admissible domains}

\begin{definition}\label{def-admsdomain}
We say that a bounded domain $\Omega$ is admissible if:\begin{enumerate}
\item[i)] it is simply connected;
\item[ii)] its boundary consists in the union of $C^2$ open arcs $\{A_i\},$ $\{B_i\}$ and $\{C_i\}$ and their endpoints, satisfying $\kappa_{euc}(A_i)=2H/y,$ $\kappa_{euc}(B_i)=-2H/y$ and $\kappa_{euc}(C_i)\ge 2H/y,$ respectively (with respect to the interior of $\Omega$) and no two of the arcs $A_i$ and no two of the arcs $B_i$ have a common endpoint;
\item[iii)] if the family $\{B_i\}$ is non empty, we require that for all $i,$ if $p_i$ and $q_i$ are initial and final points of $B_i,$ there is an arc $B_i^*\subset \hh\setminus\Omega$ of $\kappa_{euc}= 2H/y$ such that the domain $W$ bounded by $B_i\cup B_i^* \cup \{p_i, q_i\}$ is the smallest convex domain bounded by arcs of Euclidean curvature $2H/y$ (with respect to the interior of $W$) that connects $p_i$ and $q_i$.
\item[iv)] We define $\Omega^*$ as the domain obtained by replacing the boundary arcs $B_i$ of $\partial \Omega$ by $B_i^*.$ We require $\Omega^*$ to be simply connected and admit a bounded subsolution of \eqref{eq-cmc}. If $\{B_i\}$ is empty, we require $\Omega$ to admit a bounded subsolution of \eqref{eq-cmc}.
\end{enumerate}
\end{definition}

Corollary \ref{cor-exist-particular} states that, if the domain is small enough (depending on $H$), it admits a solution and hence condition iv) is satisfied.

\begin{remark}
We compare condition iii) to its analogous version in $\mathbb{R}^3$ formulated in \cite{S}: ``An arc $B_i$ of constant Euclidean curvature $2H$ is required to have length less than $\pi/2H,$ so that it consists in less than half-circle and hence its reflection along the line connecting its endpoints, together with $B_i$ bounds the smallest convex set bounded by two curves of curvature $2H,$ oriented inwards, that connect the endpoints of $B_i.$ Thus, $B_i^*$ can be taken as the reflection of $B_i$." Since in our setting we do not have reflections in all directions, this fact is no longer true, and therefore we need condition iii) in Definition  \ref{def-admsdomain}.\\
This condition is necessary in the proofs of our results, since we first establish a sequence of solutions in $\Omega^*$ and then we prove that the sequence diverges in $\Omega ^*\setminus \Omega,$ using our knowledge of the divergence set (Section \ref{sec-monotone}), which can only be applied if iii) holds.
\end{remark}

\begin{definition}[Admissible polygon] Let $\Omega$ be an admissible domain. We say that $\mathcal{P}\subset\overline{\Omega}$ is an admissible polygon if it is a polygon whose sides are curves of Euclidean curvature $\pm 2H/y,$ which we call $2H/y-$curves,
and whose vertices are chosen from among the endpoints of arcs of the families $\{A_i\}$ and $\{B_i\}.$
\end{definition}

\begin{definition} [Dirichlet problem] Let $\Omega$ be an admissible domain and fix $H>0$. The generalized Dirichlet problem is to find a solution of  (\ref{eq-cmc}) in $\Omega$ of
mean curvature $H$, which assumes the value $+\infty$ on each arc $A_i$ , $-\infty$ on each arc $B_i$ and prescribed continuous data on each arc $C_i$.
\label{def_diriinfinito}
\end{definition}

For an admissible polygon $\mathcal{P},$ we denote by $\alpha$ and $\beta$ the total Euclidean length of the arcs in $\partial \mathcal{P}$ which belong to $A_i$ and $B_i,$ respectively, by $\ell$ the perimeter of $\cal{P},$ and $\mathcal{I}(\mathcal{P})=\int_{\mathcal{P}}\frac{1}{y}{\rm d}a,$ where ${\rm d}a$ is the Euclidean area element.
We remark that the quantity $\mathcal{I}(\mathcal{P})$ has already appeared in Section \ref{sec-flux} (see equation \eqref{eq-flux1}). In fact, it corresponds to the hyperbolic integral:
$$\mathcal{I}(\mathcal{P})=\int_{\mathcal{P}}\frac{1}{y}{\rm d}a=\int_{\mathcal{P}}y{\rm d}\sigma.$$

\section{Existence Theorems - part I}

In this section we present necessary and sufficient conditions for some admissible domains to have a solution to \eqref{eq-cmc} that assigns a continuous boundary data in the arcs $C_i,$ $+\infty$ in the arcs $A_i$'s and $-\infty$ in the arcs $B_i$'s. We will prove Theorem \ref{teo-Bempty} here but in order to prove Theorem \ref{teo-Aempty} we first need to show the result for the case where $\kappa(C_i)>2H/y$ and construct barriers. The difference between the proofs of these two theorems is that on Theorem \ref{teo-Bempty} the assumption that the assigned boundary data is bounded below allows us to use a bounded subsolution as a barrier from below; however in Theorem \ref{teo-Aempty} the assigned boundary data is bounded from above and the bounded subsolution does not work, so in this case we need to construct a barrier (see Proposition \ref{prop-constBC}).

\begin{proof}[Proof of Theorem \ref{teo-Bempty}]
First suppose the inequality $2\alpha <\ell+2H\mathcal{I}(\mathcal{P})$ holds for any admissible polygon $\cal P.$

Let $u_n$ be the solution of \eqref{eq-cmc} in $\Omega$ assuming the boundary data $$u_n=\left\lbrace\begin{array}{ll}
                 n& \text{ on }A_i\\
                 \min\{f,n\}&\text{ on }C_i.
\end{array}\right.$$
Each $u_n$ exists and is unique by Theorem \ref{theo_existence}. From the General Maximum Principle (Lemma \ref{genmaxprin}), the sequence $\left\{u_n\right\}$ is monotone increasing and the second item of Lemma \ref{lem-viz} implies that $\{u_n\}$ is bounded above in a neighborhood of each arc $C_i;$ hence the Monotone Convergence Theorem (Theorem \ref{theo-monotone}) applies. 

Let $U$ be the open set in which the sequence $\{u_n\}$ converges to a solution of $(\ref{eq-cmc})$ and let $V$ be its divergence set. From Theorem \ref{theo-monotone} and Lemma \ref{lem-Vboundary}, each connected component of $V$ must be bounded by curves of $\partial\Omega$ and interior curves of $\kappa_{euc}=- 2H/y,$ if oriented to $V.$ Besides, its vertices must be among the endpoints of $\{A_i\}$ and, from Lemma \ref{lem-viz}, item \emph{ii)}, $\partial V$ cannot contain any boundary arc of the family $\{C_i\}.$ Hence any connected component of $V$ is an admissible polygon.

If $V$ is not an empty set, let $\mathcal{P}$  be a connected component of $V.$ Since each $u_n$ is a solution to \eqref{eq-cmc} in $\mathcal{P},$ equality \eqref{eq-flux1} implies that 
\begin{equation}\label{eq-desiVempt}
2H\mathcal{I}(\mathcal{P})=\int_{\mathcal{P}}2yH{\rm d}\sigma=\int_{\cup_i A_i}\langle yX_{u_n}, \nu\rangle {\rm d}s+\int_{\partial \mathcal{P}\setminus \cup_i A_i}\langle yX_{u_n}, \nu\rangle {\rm d}s,\end{equation}
where $\cup_i A_i$ takes only the $i$'s such that $A_i$ is a part of $\partial \mathcal{P}.$

From Lemma \ref{lem-flux1} and \ref{lem-flux4}, we conclude that

$$\left|\int_{\cup_i A_i}\langle yX_{u_n}, \nu\rangle {\rm d}s\right|\leq \alpha$$  and 
$$\lim_{n\rightarrow +\infty} \int_{\partial \mathcal{P}\setminus \cup_i A_i}\langle yX_{u_n}, \nu\rangle {\rm d}s =-(\ell-\alpha),
$$
which implies that $2H\mathcal{I}(\mathcal{P})\le -(\ell-\alpha)+\alpha,$ a contradiction. We then conclude that $V$ is an empty set and that $U=\Omega.$

It remains to see that $u$ takes the required boundary data.
To see that $u=f_i$ on each $C_i,$ we follow the same barrier argument as the proof of Theorem \ref{theo_existence}, noticing that for each $p\in C_i,$ the sequence $u_n$ is uniformly bounded in a neighborhood of $p,$ a consequence of Lemma \ref{lem-viz} and the fact that $\Omega$ admits a global subsolution to the Dirichlet Problem.

Let us notice that $u$ indeed goes to $+\infty$ on the arcs $\{A_i\}.$
Recall that $u_0$ is a global solution to the Dirichlet Problem with $u_0\leq u_n$ in $\Omega.$ Given $p\in A_i,$ let $U_p$ be a neighborhood of $p$ for which the existence result (Theorem \ref{theo_existence}) applies. Take $v_k$ the solution to the Dirichlet problem in $U_p\cap \Omega$ with boundary data $m=\min_U {u_0}$ on $\partial U_p \cap \Omega$ and $k$ on $\partial \Omega \cap U_p.$ Then for $n\geq k,$ $u_n\geq v_k$ from the Maximum Principle. Therefore, $u\geq v_k$ for all $k$ and $u$ diverges to $+\infty$ on $U_p\cap \partial \Omega.$

Reciprocally, assume that the Dirichlet problem has a solution in $\Omega.$ We know the equality \eqref{eq-flux1} holds in any admissible polygon $\cal P$ in $\overline{\Omega}.$

By Lemma \ref{lem-flux3}, we have $$\int_{\cup_i A_i}\langle yX_{u}, \nu\rangle {\rm d}s=\alpha.$$

Using Lemma \ref{lem-flux1} for the curves $\partial \mathcal{P}\setminus \cup_i A_i$ that are not on $\partial \Omega$ and Lemma \ref{lem-flux2} for the ones on $\partial\Omega,$ we get
$$
\int_{\partial \mathcal{P}\setminus \cup_i A_i}\langle yX_{u}, \nu\rangle {\rm d}s
 >-(\ell-\alpha).
$$

Then,
 $$2H\mathcal{I}(\mathcal{P})=\int_{\cup_i A_i}\langle yX_{u}, \nu\rangle {\rm d}s+\int_{\partial \mathcal{P}\setminus \cup_i A_i}\langle yX_{u}, \nu\rangle {\rm d}s
 >-(\ell-\alpha)+\alpha,$$
and therefore, for any admissible polygon, we have
$$
2\alpha <\ell+2H\mathcal{I}(\mathcal{P}).
$$

\end{proof}

\begin{proposition}\label{prop-Aempty}
Let $\Omega$ be an admissible domain such that the family $\{A_i\}$ is empty. Assume that $\kappa_{euc}(C_i)>2H/y$ and that the assigned boundary data on the arcs $\{C_i\}$ is bounded above. Then the Dirichlet Problem has a solution in $\Omega$ if and only if
$$2\beta <\ell-2H\mathcal{I}(\mathcal{P})$$
for all admissible $\mathcal{P}.$
\end{proposition}

The proof of this proposition in Sol$_3$ is analogous to the proof of this result in other ambient spaces (see for instance Proposition 7.2 in \cite{FR}) using the same adaptations used in the above result.

\section{More on admissible domains}

In this section we state and prove some properties of the $2H/y-$curves in order to construct admissible domains around any point of a $2H/y-$curve. These domains will be constructed (in Section 11) so that they satisfy the hypothesis of Theorem \ref{teo-Bempty} and Proposition \ref{prop-Aempty} and therefore will admit solutions to the Dirichlet Problem. Later we will use these surfaces as barriers to improve Proposition \ref{prop-Aempty} for the case of $\kappa_{euc}(C_i)\geq 2H/y$ and then prove Theorem \ref{teo-Aempty}. 

The first result about these curves presents their shape and parametrization. These curves were also exhibited in \cite{LM} in the Lie Group model of Sol$_3.$

\begin{proposition}\label{prop-eqcurv2H/y}
An arc of Euclidean curvature $2H/y$ is part of the trace of a curve $\gamma_P,$ $P=(w,z)\in \mathbb{R}^2_+,$  parametrized by $t\mapsto (x(t),y(t)),$ $t\in \rr,$ where

$$\left\lbrace\begin{array}{l}
\displaystyle{x(t)=w+\frac{ze^{1/{2H}}}{2H}S_{H}(t)} \\
\\
\displaystyle{y(t)=ze^{\frac{\sin^2(t/2)}{H}}.}
\end{array}\right.$$

The function $S_H$ is defined by
$$S_H(t)=\left\lbrace\begin{array}{l}
\displaystyle{\int_{-1}^{-\cos t} \frac{-ue^{u/2H}}{\sqrt{1-u^2}}du,}\; t\in [-\pi,0]\\ \\
\displaystyle{\int_{-1}^{-\cos t} \frac{ue^{u/2H}}{\sqrt{1-u^2}}du,}\; t\in [0,\pi],
\end{array}\right.$$ for $t\in [-\pi,\pi],$ and is extended to $\mathbb{R}$ using the relation
$S(2n\pi+t_0)=2nS(\pi)+S(t_0),$ $t_0\in [-\pi,\pi],$ $n\in \mathbb{Z}.$
\end{proposition}

\begin{proof}
A curve in $\mathbb{R}^2_+,$ parameterized by $\gamma(t)=(x(t),y(t))$ has Euclidean curvature $\kappa_{euc}=-2H/y,$ $H\neq 0,$ if and only if
\begin{equation}\label{eq_curva}
x'(t)y''(t)-x''(t)y'(t)=\frac{-2H}{y(t)}(x'(t)^2+y'(t)^2)^{3/2},
\end{equation}
which holds for $\gamma_P$ defined above.
\end{proof}

\begin{remark}
A clearer parametrization of $\gamma_P$ that sees it as a graph in the horizontal direction and does not take into account the sign of $\kappa_{euc}$ is $\gamma_P(y)=(x(y),y),$ $y\in \left(z,ze^{2/{2H}}\right),$ where 
\begin{equation}\label{eq-curva-x=x(y)}
x(y)=w\pm\frac{ze^{1/2H}}{2H}\int_{-1}^{-1+2H\ln(y/z)} \frac{-ue^{u/2H}}{\sqrt{1-u^2}}du.
\end{equation}

We name the function in the integrand by \begin{equation}\label{eq-defdag}
g_H(u) =\frac{-ue^{u/2H}}{\sqrt{1-u^2}}.
\end{equation}

\end{remark}

The shape of a $2H/y-$curve (see Figure \ref{figure1}) repeats periodically if one moves horizontally. We define some useful quantities associated to these curves.

\begin{definition}
For any $H>0,$ we define
\begin{equation} \label{eq-L}
L_H=\frac{e^{1/2H}}{2H}\int_{-1}^0g_H(u)du
\end{equation}

\begin{equation} \label{eq-M}
M_H=\frac{e^{1/2H}}{2H}\int_{-1}^1-g_H(u)du
\end{equation}

and $T=T_H$ as the number in the interval $(0,1)$ such that
\begin{equation} \label{eq-T}
\int_{-1}^{T_H}g_H(u)du=0.
\end{equation}

\label{def-T}

\end{definition}

\begin{figure}[h]
  \centering
  \includegraphics[height=5cm]{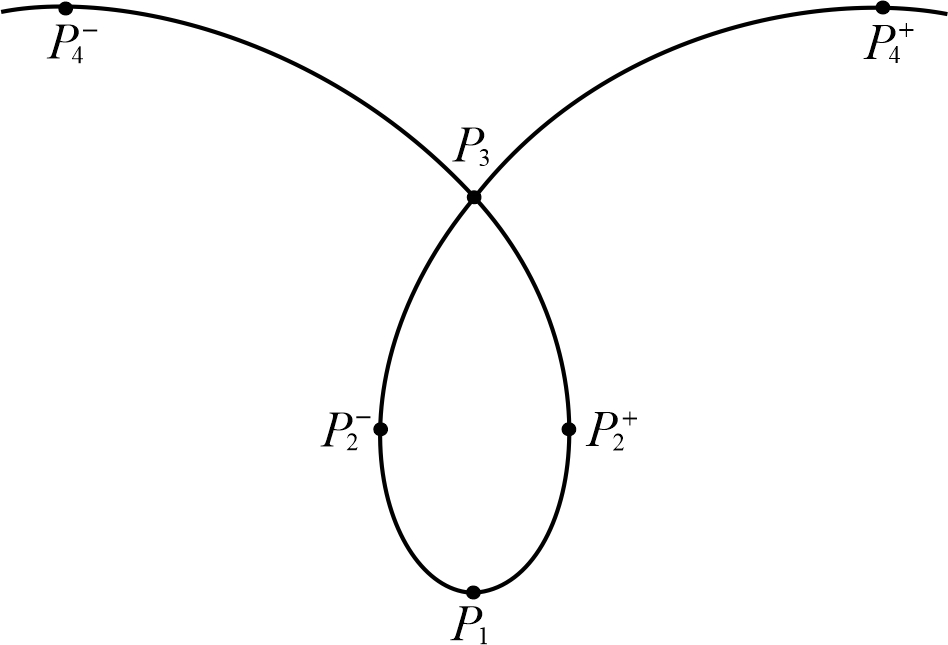}
  \caption{Curve of Euclidean curvature $2H/y.$}
    \label{figure1}
\end{figure}

We can see from Proposition \ref{prop-eqcurv2H/y} that if the base point $P=P_1$ has coordinates $P_1=(0,z),$ the other indicated points in the figure have the following coordinates:

$$ \begin{array}{lll}

P_2^+&=\gamma_P\left(-\pi/2\right)&=\left(zL_H,ze^{1/2H}\right),\\
P_2^-&=\gamma_P\left(\pi/2\right)&=\left(-zL_H,ze^{1/2H}\right),\\
P_4^{+}&=\gamma_P\left(\pi\right)&=\left(zM_H,ze^{1/H}\right),\\
P_4^{-}&=\gamma_P\left(-\pi\right)&=\left(-zM_H,ze^{1/H}\right),\\
P_3&=\gamma_P\left(T_H\right)&=\left(0, ze^{\frac{1+T_H}{2H}}\right).\\

\end{array}$$

Since the coordinates are all proportional to $z,$ we conclude that moving upwards makes the period and height of the $2H/y-$curves grow linearly. Since an Euclidean dilation centered at the origin is an hyperbolic isometry in this model, this shows that in $\mathbb{H}^2$ all the $2H/y-$curves are isometric.

\subsection{$2H/y-$curves joining aligned points}
To build the admissible domains around any point in a $2H/y-$curve that satisfy the conditions of Theorem \ref{teo-Bempty} and Proposition \ref{prop-Aempty}, we must analyse all inscribed polygons, that is, all $2H/y-$curves contained in the domain that connect the vertices of the boundary of the domain. In order to simplify this analysis we construct (Section 11) admissible domains using four vertices of a rectangle with sides parallel to the $x$ and $y$ axes.

To verify the hypotheses of Theorem \ref{teo-Bempty} and Proposition \ref{prop-Aempty} we need the next results about $2H/y-$curves connecting horizontally or vertically aligned points.

Given two horizontally aligned points $p=(-w,z)$ and $q=(w,z)$ we would like to know what are all embedded $2H/y-$curves from $p$ to $q.$ The next proposition gives sufficient conditions on the Euclidean distance $2w$ between $p$ and $q$ for the existence of at most three embedded $2H/y-$curves connecting them.

\begin{proposition}\label{prop-horizalig}
Given two horizontally aligned points $p=(-w,z)$ and $q=(w,z)$ with Euclidean distance $2w<zK(H),$ any embedded $2H/y-$curve from $p$ to $q$ is:
\begin{description}
\item[i)] symmetric about reflection on the vertical line $x=0,$
\item[ii)] either contained in the region $\{y\geq z\}$ or contained in the region $\{y\leq z\}.$
\end{description}
The constant $K=K(H)$ depends only on $H$ and is given by expression \eqref{eq-K(H)} below.

Moreover, there is exactly one such curve above the line $y=z$ and there are at most two such curves below the line $y=z.$ If $2w<2ze^{-1/2H}L_H,$ there are two curves below the line $y=z$, one which is shorter, without point of vertical tangency, and another one that has two points of vertical tangency (that is, it has the two $P_2^{-}$ and $P_2^{+}$ type of points, see Figure \ref{figure1}); in particular its $P_1$-type of point is in the line $y=ze^{-1/2H}$ or below it.\\

\end{proposition}

\begin{proof}
All possible $2H/y-$curves connecting $p$ to $q$ are parts of a curve $\gamma_P$ described in Proposition \ref{prop-eqcurv2H/y}.
Hence their possible shapes are obtained intersecting $\gamma_P$ with horizontal lines.

To look at all possible intersections we move downwards a horizontal line and look at parts of $\gamma_P$ that join two intersection points (horizontally aligned) and are embedded. We first find curves above the horizontal line connecting the two intersection points which exist if $$2w=d_{euc}(p,q)\leq 2ze^{\frac{-1}{2H}}(M_H+L_H),$$ where the equality corresponds to the case where $p$ is a $P_2^{-}$-point and $q$ is a $P_2^+$-point from the next loop of $\gamma_P$.

If $p$ is a point in the arc $\arc{P_3P_2^+P_1}$ of a loop of a curve $\gamma_P$ and $q$ is a point in the arc $\arc{P_1P_2^-P_3}$ of the next loop of $\gamma_P,$ the part of $\gamma_P$ joining the two points is an embedded arc but is not contained in $\{y\geq z\}$ nor in $\{y\leq z\}.$ We want to avoid this kind of curves.  

In order to do that we fix the line $y=z$ and look at the family of all $2H/y$-curves that intersects this line and have the $P_1$-type point below it. We define $d(t)=w_2(t)-w_1(t),$ for any $t\in [0, 1+T_H],$  the Euclidean distance between $(w_1(t),z)$ and $(w_2(t),z),$ the first and second intersections of $\gamma_{(0, ze^{-t/2H})}$ with the half line $\{(x,y)\,|\,y=z \text{ and }x\geq 0\}.$ (Observe that since translations along the $x$-direction are isometries, we can restrict ourselves to the family of $2H/y$-curves whose $P_1$-type point is of the form $(0, ze^{-t/2H})$).

We compute $w_1$ and $w_2$ and then look for the minimum value of $d.$
On one hand, $w_1(t)$ is the $x$-coordinate of $\gamma_{(0, ze^{-t/2H})}$ when the $y$-coordinate is $z,$ then $$w_1(t)=\frac{ze^{-t/2H}e^{1/2H}}{2H}\int_{-1}^{-1+t} g_H(u) du.$$
On the other hand, $w_2(t)$ is the distance between $(0, ze^{-t/2H})$ and the next $P_1$-point of $\gamma_{(0, ze^{-t/2H})},$ which is $2M_Hze^{-t/2H}$, minus $w_1(t),$ hence
$$w_2(t)=2M_Hze^{-t/2H}-w_1(t)$$ and then
$$\begin{array}{ll}
d(t)&\displaystyle{=2M_Hze^{-t/2H}-2w_1(t)}\\ \\
&=\displaystyle{2ze^{-t/2H}\left(M_H-\frac{e^{1/2H}}{2H}\int_{-1}^{-1+t} g_H(u) du\right)}\\ \\
&=\displaystyle{z\bar{d}(t).}
\end{array}$$

We notice that $\bar{d}$ is a continuous function in $[0,1+T_H]$ that depends only on $H$ and we look for the minimum value of $\bar{d}.$ We claim that there is a point $t_0\in (0,1+T_H)$ such that $\bar{d}$ is decreasing in $[0,t_0]$ and increasing in $[t_0,1+T_H].$ Therefore, $\bar{d}(t_0)$ is the minimum value of $\bar{d}.$

Differentiating $\bar{d},$ we have
$$\bar{d}'(t)=2e^{-t/2H}\left(\frac{-M_H}{2H}+\frac{e^{1/2H}}{(2H)^2}\int_{-1}^{-1+t} g_H(u) du-\frac{e^{1/2H}}{2H}g_H(-1+t)\right),$$
which has the same sign as 
$$f(t)=\frac{-M_H}{2H}+\frac{e^{1/2H}}{(2H)^2}\int_{-1}^{-1+t} g_H(u) du-\frac{e^{1/2H}}{2H}g_H(-1+t).$$

Differentiating $f$ we have
$$f'(t)=\displaystyle{\frac{e^{{t}/{2H}}}{2H(1-(-1+t)^2)^{3/2}}>0.}$$

Moreover,
$$\lim_{t\rightarrow0^+}f(t)=-\infty$$
and $f$ can be extended continuously to $[0,2],$ and then
$$\lim_{t\rightarrow 2^-} f(t)=\frac{-M_H}{2H}+\frac{e^{1/2H}}{(2H)^2}\int_{-1}^{1} g_H(u) du-\lim_{t\rightarrow 1^-}\frac{e^{1/2H}}{2H}g_H(t)=+\infty.$$

Hence $f$ takes the value zero in $(0,2)$ and therefore $\bar{d}'$ too. It is easy to see from the geometric definition of $\bar{d},$ which also can be extended to $[0,2],$ that $t_0\in (0,1+T_H).$

Therefore $t_0$ is the unique solution of $f(t_0)=0,$ that is,

$$\frac{M_H}{2H}=\frac{e^{1/2H}}{(2H)^2}\int_{-1}^{-1+t_0} g_H(u) du-\frac{e^{1/2H}}{2H}g_H(-1+t_0).$$

Let us define
\begin{equation}\label{eq-K(H)}
K(H)=\bar{d}(t_0)=2e^{-t_0/2H}\left(M_H-\frac{e^{1/2H}}{2H}\int_{-1}^{-1+t_0} g_H(u) du\right)
\end{equation}
for $t_0$ the solution above.

Therefore, if $p$ and $q$ are closer than $zK(H),$ any $2H/y-$curve connecting them is either above the line $y=z$ as described in the beginning of this proof or below it. 

As a consequence of the Maximum Principle, we can show that the curve above is unique. In fact, let us denote by $p$ and $q$ two horizontally aligned points whose distance is less than $zK(H)$. Then any $2H/y-$curve joining them that is above the line $y=z$ does not contain any point with vertical tangency, its interior is contained in the open vertical slab bounded by the vertical lines passing through $p$ and $q$, and it is symmetric with respect to the vertical line passing through the midpoint between $p$ and $q$. In particular, if there were two such curve, say $c_1$ and $c_2$, they would be disjoint apart from the endpoints and one would be above the other (see Figure \ref{curves1}). Hence moving $c_1$ farway upwards and translating it back, we would find a first point of contact between $c_1$ and $c_2$, but since both of them satisfy that at each point their Euclidean geodesic curvature is $2H/y$ we would get a contradiction with the Maximum Princeple, because the curve $c_1$ was below $c_2$ and therefore its $y$-point is smaller than the corresponding $y$-point of $c_2$, making the geodesic curvature of $c_2$ be less than the geodesic curvature of $c_1$. Therefore, there is only one curve above.

\begin{figure}[h]
\centering
\includegraphics[height=1.5cm]{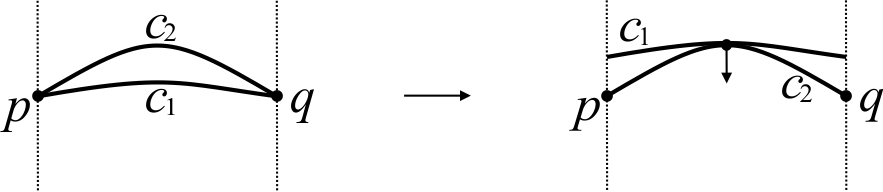}
\caption{Horizontally aligned points.}
\label{curves1}
\end{figure}

The case of curves below $y=z$ is different, we might not have uniqueness. By looking at the intersections of $\gamma_P$ with horizontal lines, it seems that for $p$ and $q$ sufficiently close there are usually two curves connecting $p$ to $q$ below $y=z.$ We will show this is always the case.

To demonstrate that, we again analyse the distance between two intersections of a line $y=z$ and a family of curves $\{\gamma_{(0, ze^{-t/2H})}\}_{t\in (0, 1+T_H)}.$ The difference is that now we are interested in the distance $l(t)$ between the first intersection of $\gamma_{(0, ze^{-t/2H})}$ with $\{y=z, x\geq0\}$ and the last intersection of $\gamma_{(0, ze^{-t/2H})}$ with  $\{y=z, x\leq 0\}.$

As done before, we compute \begin{equation}\label{eq-defl}
l(t)=2w_1(t)=\frac{2ze^{-t/2H}e^{1/2H}}{2H}\int_{-1}^{-1+t} g_H(u) du=z\bar{l}(t).
\end{equation}

Again differentiating and in this case knowing that $l(0)=l(1+T_H)=0,$ we find a value $t_0\in (0,1+T_H)$ such that $\bar{l}$ is increasing in $[0,t_0]$ and decreasing in $[t_0, 1+T_H],$ with a maximum value $\bar{l}(t_0),$ and for each distance between $0$ and $z\bar{l}(t_0),$ there are two $2H/y-$curves joining a pair of points that are this distance apart.

We remark that $\bar{l}$ is $e^{-t/2H}$ multiplied by a function that attains its maximum at $t=1;$ 
hence, the maximum value of $\bar{l}$ occurs in $t_0\in (0,1)$ and $\bar{l}$ is decreasing in $(1,1+T_H).$ Therefore, given $2w \in (0, z\bar{l}(1)),$ the curves that join two points $2w$ apart are one with $P_1$-point above $y=ze^{-1/2H}$ and the other with $P_1$-point below $y=ze^{-1/2H}.$ 
Since $$\bar{l}(1)=2e^{-1/2H}L_H,$$ the proof is concluded.
\end{proof}

The next result is about a domain $\Omega$ with only two horizontally aligned vertices $p$ and $q$ contained in the line $y=z.$ 

Let us call $\gamma^+$ the curve that joins $p$ and $q$ contained above the line $y=z$ and by $\gamma^-$ the curve below joining $p$ and $q$ without point of vertical tangency. If $\ell_{euc}(\gamma^{\pm})$ stands for the Euclidean length of $\gamma^{\pm},$ observe that 
$$
\ell_{euc}(\gamma^+)<\ell_{euc}(\gamma^-)+2H\mathcal{I}(\Omega),
$$
since $\ell_{euc}(\gamma^+)<\ell_{euc}(\gamma^-),$ because $\gamma^-$ has greater Euclidean curvature than $\gamma^+.$ Moreover, we have the following lemma.

\begin{lemma}\label{lemma-horizdesig}
In the context described above, let $\Omega$ be the domain bounded below by $\gamma^-$ and above by $\gamma^+.$ Then, if $p=(-w,z)$ and $q=(w,z)$ are sufficiently close, it holds $$\ell_{euc}(\gamma^-)<\ell_{euc}(\gamma^+)+2H\mathcal{I}(\Omega).$$
\end{lemma}

\begin{proof}
This inequality $\ell_{euc}(\gamma^-)<\ell_{euc}(\gamma^+)+2H\mathcal{I}(\Omega)$ is a consequence of the more general one
\begin{equation}\label{eq-desiglema}
\ell_{euc}(\gamma^-)<d_{euc}(p,q)+2H\mathcal{I}(\Omega^-),
\end{equation} where $\Omega^-$ is the region bounded below by $\gamma^-$ and above by the straight line $\overline{pq}.$ We will show that this inequality holds if the base point $P_1=(0,ze^{-a/2H})$ of $\gamma^-$ is such that $a+e^{a/2H}<2,$ the reason why $p$ and $q$ are required to be sufficiently close.

In order to do that, we compute and estimate $\ell_{euc}(\gamma^-)-d_{euc}(p,q)$ and $\mathcal{I}(\Omega^-).$ We have
$$\ell_{euc}(\gamma^-)=2\frac{ze^{-a/2H}e^{1/2H}}{2H}\int_{-1}^{-1+a} \frac{e^{u/2H}}{\sqrt{1-u^2}}du$$
and $$d_{euc}(p,q)=2w=2\frac{ze^{-a/2H}e^{1/2H}}{2H}\int_{-1}^{-1+a} g_H(u)du,$$
so, using that the exponential function is increasing, we get
$$\begin{array}{ccl}
\ell_{euc}(\gamma^-)-d_{euc}(p,q)&=&\displaystyle{\frac{ze^{-a/2H}e^{1/2H}}{H}\int_{-1}^{-1+a} e^{u/2H}\sqrt{\frac{1+u}{1-u}}du}\\ \\
&\leq& \displaystyle{\frac{z}{H}\int_{-1}^{-1+a} \sqrt{\frac{1+u}{1-u}}du} = \displaystyle{\frac{z}{H}\int_{0}^{a} \sqrt{\frac{t}{2-t}}dt}.
\end{array}$$

On the other hand, for $x_{\gamma^-}$ the $x$-coordinate of $\gamma^-$,
$$\begin{array}{ccl}
2H\mathcal{I}(\Omega^-)&=&\displaystyle{4H\int_{ze^{-a/2H}}^{z}\frac{x_{\gamma^{-}}(y)}{y}dy}=\displaystyle{\frac{ze^{-a/2H}e^{1/2H}}{H}\int_0^a \int_{-1}^{-1+t}
g_H(u)du dt}\\\\
&\ge& \displaystyle{\frac{ze^{-a/2H}}{H}\int_0^a \int_{-1}^{-1+t}
\frac{-u}{\sqrt{1-u^2}}du dt}=\displaystyle{\frac{ze^{-a/2H}}{H}\int_0^a \sqrt{2t-t^2} dt},
\end{array}$$
where in the second equality we did the change of variables $y=ze^{t-a/2H}$ and used \eqref{eq-curva-x=x(y)},  and in the inequality we used again that the exponential function is increasing.

Hence \eqref{eq-desiglema} is implied by $$\frac{ze^{-a/2H}}{H}\int_0^a \sqrt{2t-t^2} dt \geq \frac{z}{H}\int_{0}^{a} \sqrt{\frac{t}{2-t}}dt,$$
which is a consequence of $$\sqrt{2t-t^2}  \geq e^{a/2H} \sqrt{\frac{t}{2-t}}$$
that holds for $t\in (0,a),$ if $a$ satisfies $a +e^{a/2H}<2.$ 
\end{proof}

The constructions of domains around points in $2H/y-$curves also need some properties of $2H/y-$curves that join vertically aligned points.

\begin{proposition}\label{prop-vertalig}
Given two vertically aligned points $p=(0,z)$ and $q=(0,t)$ with Euclidean distance $0<t-z<z(e^{T_H/2H}-1),$ 
there are two $2H/y-$curves $\gamma_I$ and $-\gamma_I$ joining $p$ and $q$ contained in the slab $\{z\leq y\leq t\};$ one of them is contained in $\{x\geq 0\}$ and the other one is the reflection of the first about $x=0.$ Moreover, any other $2H/y-$curve with endpoints $p$ and $q$ intersects the line $y=ze^{-1/2H},$ intersects both half planes $\{x<0\}$ and $\{x>0\},$ and has length greater than the length of $\gamma_I.$ 
\end{proposition}

The fact that other curves intersect the line $y=ze^{-1/2H}$ is important because it allows the construction of domains that do not contain these curves in its interior and, therefore, we do not need to care about their length.

\begin{proof}
The proof follows the same initial steps as the one for horizontally aligned points. 
Since the reflection about a vertical line is an isometry, any curve considered has its reflected correspondent, which we omit for shortness.
Intercepting a $2H/y-$curve with a vertical line, we realize that there are three possible types of curves joining points $p$ and $q$ vertically aligned:
\begin{description}
\item Type I: Occurs when $p$ is between $P_1$ and $P_2^+$ and $q$ is between $P_2^+$ and $P_3$ and looks like the figure of letter D without the vertical segment. This is the only type of curve contained in the horizontal slab determined by the two horizontal lines passing through $p$ and $q.$
\item Type II: Occurs when $p$ is between $P_1$ and $P_2^-$ and $q$ is between $P_3$ and $P_4^-.$ 
\item Type III:  Occurs when $p$ is between $P_2^-$ and $P_3$ and $q$ is between $P_3$ and $P_4^-.$ 
\end{description}

Notice that if $p$ and $q$ are endpoints of a curve of Type II, then its $P_1$-point, $P_1=(x_1,y_1),$
has the following relations with the coordinates of $p$ and $q:$
$$y_1<z<y_1e^{1/2H} \text{ (because $p$ is between $P_1$ and $P_2^-$) and} $$
$$y_1e^{(1+T_H)/2H}<t<y_1e^{2/2H}\text{ (because $q$ is between $P_3$ and $P_4^-).$}$$

Hence
$z<y_1e^{1/2H}<te^{-T_H/2H},$ implying $ze^{T_H/2H}<t$ and
$t-z>z(e^{T_H/2H}-1).$ 
We may conclude that if $t-z<z(e^{T_H/2H}-1)$ then no curve of Type II from $p=(0,z)$ to $q=(0,t)$  exists. 

Notice also that a curve of Type III always intersects the line $y=ze^{-1/2H}$ because $p$ is above its $P_2^+$-point and, therefore, its $P_1$-point, $P_1=(x_1,y_1),$ is such that $y_1e^{1/2H}<z.$

Having this observed we prove that the curves of Type I exist and are unique in each side of the vertical line.
The existence is a consequence of $t-z<z(e^{T_H/2H}-1)<z(e^{(1+T_H)/2H}-1)=d(P_1,P_3),$ for $P_1=(0,z).$
The uniqueness is a consequence of the Maximum Principle. In fact, if there were two $2H/y-$curves from $p$ to $q$ in the same side of the line $x=0,$ by moving horizontally one of them we would find a last interior contact point, which would be a tangency point. From the Maximum Principle, the translation of one curve would coincide with the other. Since their endpoints are the same, they have to be the same.

It remains to show that a curve of Type III, denoted by $\gamma_{III},$ from $p$ to $q$ is longer than  a $2H/y-$curve of Type I from $p$ to $q.$ Let us assume without loss of generality that $\gamma_I$ is contained in $\{x\geq 0\}$ and $\gamma_{III}$ contains points with $y$-coordinate greater than $z$ in $\{x\geq 0\}.$ Hence $\gamma_{III}$ is contained in the complement of the open region $O_I$ bounded by $\gamma_I$ and $-\gamma_I,$ the reflection of $\gamma_I$ about the vertical line $x=0.$ In fact, if it were not so, there would be a region in the slab $\{z\leq y\leq t\}$ bounded on the left by an arc of $\gamma_{III}$ and on the right by an arc of $\gamma_I,$ both with mean curvature vectors pointing to the left; then moving $\gamma_{III}$ to the right there would be a last interior contact point contradicting the Maximum Principle.
Therefore, since $\gamma_{III}$ is outside the convex domain $O_I,$ it is longer than its projection on $O_I,$ defined by $\pi(v)=\{u\in O_I| d(u,v)\leq d(u,w), \, \forall w\in O_I\},$ which is $\gamma_I$ together with a part of $-\gamma_I,$ in particular, it is longer than $\gamma_I.$

\end{proof} 

The last result of this subsection is about domains bounded by a $2H/y-$curve of Type I and its reflection.

\begin{lemma}\label{lem-vertalign}
Let $p=(0,z)$ and $q=(0,t)$ be two vertically aligned points. If $\gamma_I$ and $-\gamma_I$ are the $2H/y-$curves joining $p$ and $q$ of Type I and $\Omega$ is the region between them, then
if $\ell$ denotes the Euclidean length of $\gamma_I$ (the same of $-\gamma_I$), it holds
$$2\ell>2H\mathcal{I}(\Omega).$$
\end{lemma}

\begin{proof}
By prolonging $\gamma_I$, we find the $y-$coordinate of its $P_1-$point, say $P_1=(x_1,y_1).$
Then there are $a \in (0,1)$ and $b \in (1,1+T_H)$ such that
$z=y_1e^{a/2H}$ and $t=y_1e^{b/2H}.$ Moreover, the $P_2^+-$point of $\gamma_I$ is $P_2=(x_\gamma(y_1e^{1/2H}),y_1e^{1/2H})$ for $x_\gamma$ such that $\gamma_I(y)=(x_\gamma(y),y).$

The bottom part of $\gamma_I$ from $p$ to $P_2$ has length

$$\ell_b=\left(\frac{y_1e^{1/2H}}{2H}\right)\int_{-1+a}^0 \frac{e^{u/2H}}{\sqrt{1-u^2}}du$$
computed as in the proof of Lemma \ref{lemma-horizdesig}. The upper part has length 
$$\ell_u=\left(\frac{y_1e^{1/2H}}{2H}\right)\int_0^{-1+b} \frac{e^{u/2H}}{\sqrt{1-u^2}}du.$$

The bottom part of $\Omega$, denoted by $\Omega_b,$ bounded below by the bottom parts of $-\gamma_I$ and $\gamma_I$
and above by the line $y=y_1e^{1/2H}$ satisfies
$$\begin{array}{ll}
2H\mathcal{I}(\Omega_b)&=\displaystyle{2\left(2H\int_{y_1e^{a/2H}}^{y_1e^{1/2H}}\frac{x(y)-x(y_1e^{a/2H})}{y}dy\right)} \\ \\
&\leq \displaystyle{4H\int_{y_1e^{a/2H}}^{y_1e^{1/2H}}\frac{x(y_1e^{1/2H})-x(y_1e^{a/2H})}{y}dy},
\end{array}$$
because $x_\gamma(y)\leq x(y_1e^{1/2H})$ for all $y\in \left(y_1e^{a/2H},y_1e^{1/2H}\right).$
Hence
$$\begin{array}{ll}
2H\mathcal{I}(\Omega_b)& \leq \displaystyle{4H\frac{(1-a)}{2H}\left(x(y_1e^{1/2H})-x(y_1e^{a/2H})\right)}\\ \\
& =\displaystyle{2(1-a)\left(\frac{y_1e^{1/2H}}{2H}\right)\int_{-1+a}^0 \frac{-ue^{u/2H}}{\sqrt{1-u^2}}du},
\end{array}$$
and analogously for the upper part of $\Omega,$ denoted by $\Omega_u:$
$$2H\mathcal{I}(\Omega_u)\leq 2(b-1)\left(\frac{y_1e^{1/2H}}{2H}\right)\int_{0}^{b-1} \frac{ue^{u/2H}}{\sqrt{1-u^2}}du.$$

Therefore, since $(1-a)(-u)<1$ for any $u \in (a-1,0),$ we get
\begin{equation}\label{eq-desig1}
\begin{array}{ll}
\displaystyle{2\ell_b}&=\displaystyle{2\left(\frac{y_1e^{1/2H}}{2H}\right)\int_{-1+a}^0 \frac{e^{u/2H}}{\sqrt{1-u^2}}du}\\ \\
&>\displaystyle{ 2(1-a)\left(\frac{y_1e^{1/2H}}{2H}\right)\int_{-1+a}^0 \frac{-ue^{u/2H}}{\sqrt{1-u^2}}du \geq 2H\mathcal{I}(\Omega_b).}
\end{array}
\end{equation}

The same holds for the upper part of $\gamma_I$ because $(b-1)(u)<1$ for any $u \in (0,b-1):$
\begin{equation}\label{eq-desig2}
\begin{array}{ll}
\displaystyle{2\ell_u}&=\displaystyle{2\left(\frac{y_1e^{1/2H}}{2H}\right)\int_0^{-1+b} \frac{e^{u/2H}}{\sqrt{1-u^2}}du}\\ \\
&>\displaystyle{2(b-1)\left(\frac{y_1e^{1/2H}}{2H}\right)\int_{0}^{b-1} \frac{ue^{u/2H}}{\sqrt{1-u^2}}du\geq 
2H\mathcal{I}(\Omega_u).}
\end{array}
\end{equation}

Hence, adding inequalities \eqref{eq-desig1} and \eqref{eq-desig2}, the proof is concluded.
\end{proof}

\section{Construction of some admissible domains}

\label{existencedomains}

%\pati{In this Section we exhibit some examples of admissible domains whose constrution is ilustrated in Figure \ref{fig_2}}

%\begin{figure}[h]
  %\centering
  %\includegraphics[height=6cm]{domainAempty.png}
  %\caption{Admissible domain with $\{B_i\}$ on the left and empty $\{A_i\}$ on the right}
   % \label{fig_2}
%\end{figure}

%\ana{quebrei a figura que vc fez e coloquei no meio da prova. O que acha?}

\subsection{Admissible domains with $\{B_i\}=\emptyset$}

\begin{proposition}\label{prop-consthoro}
There are domains $U$ where the existence result from Theorem \ref{teo-Bempty} holds, that is, the family $\{B_i\}$ is empty and $2\alpha< \ell+2H\mathcal I(P)$, for all admissible polygons $\mathcal P$ in $U.$
\end{proposition}

We exhibit some of these domains which are bounded by 2 arcs of circles $C_1$ and $C_2$ of $\kappa_{euc}(C_i)>2H/y$ and two arcs $A_1$ and $A_2$ of Euclidean curvature $2H/y,$ if oriented inwards.

\begin{proof}
Let $R$ be a rectangle in $\mathbb{H}^2$ such that:
\begin{enumerate}
\item[-] the sides of $R$ are parallel to the $x$ and $y$ axes;
\item[-] Denoting by $d$ the Euclidean length of the diagonal of $R$ and by $p=(x(p),y(p))$ the Euclidean center of $R$, it holds
\begin{equation}\label{eq-desigd}
d<\frac{2y(p)}{3+2H};
\end{equation}
\item[-] $R$ is sufficiently small so that a subsolution exists (see Corollary \ref{cor-exist-particular}). 
\end{enumerate}

Let $q_1$ and $q_2$ be the two endpoints of a diagonal of $R$ with $x(q_1)<x(q_2)$ and $y(q_1)<y(q_2).$ 
Let $s_1$ and $s_2$ be two vertical segments of length $\varepsilon<y(q_1)(e^{T_H/2H}-1)$ whose medium points are $q_1$ and $q_2,$ respectively, where $T_H$ is the number described in Definition \ref{def-T}. For $i=1,2,$ we replace $s_i$ by $A_i,$ an arc of Type I (see the proof of Proposition \ref{prop-vertalig}) of a $2H/y-$curve that joins the endpoints of $s_i,$ is oriented in the direction of $p$ and has length less than $d/2.$ This is possible if $\varepsilon$ is taken sufficiently small. These two arcs will be part of $\partial U.$

To build the remaining part of $\partial U$ we connect the lowest endpoint of $s_1$ to the lowest endpoint of $s_2$ by a semicircle $C_1$ whose diameter is the straight segment that joins these two points. The arc $C_2$ is built analogously by connecting the upper endpoints of $s_1$ and $s_2.$ Observe that $C_1$ and $C_2$ are semicircles with same diameter. 

%are chosen so that  we get a circle when moving $C_1$ $\varepsilon$-upwards and joining it to $C_2.$ 

Our domain $U$ is the region bounded by the arcs $A_1, A_2, C_1$ and $C_2.$

\begin{figure}[h]
\centering
\includegraphics[width=6cm]{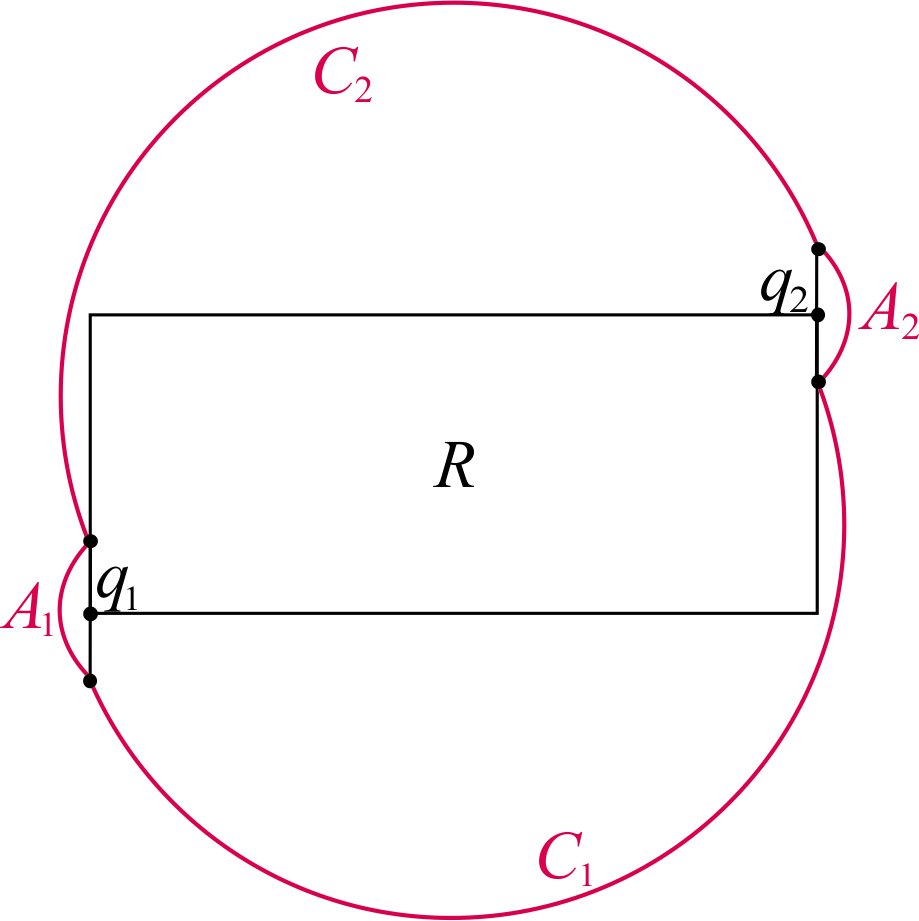}
\caption{Construction of the domain $U$ with empty $\{B_i\}$.}
\end{figure}

{\it Claim}: $U$ is an admissible domain.

Notice that
$$\kappa_{euc}(C_i)=\frac{2}{d}> \frac{2H}{y_{min}} \geq \frac{2H}{y},$$
where $y_{min}=\min \{y(q); q\in \bar{U}\}>y(p)-d-d/2$ and the first inequality is a consequence of \eqref{eq-desigd}, since it implies
$$H+\frac{3}{2}<\frac{y(p)}{d} \text{ and therefore } H<\frac{y(p)}{d}-\frac{3}{2}\leq \frac{y_{min}}{d}.$$

{\it Claim}: $U$ admits a solution to the Dirichlet problem (Definition \ref{def_diriinfinito}).

We need to check that for any admissible polygon $\mathcal{P}\subset \overline{U},$ the inequality  $2\alpha<\ell+2H\mathcal{I}(\mathcal{P})$ holds. Recall $\ell$ denotes the Euclidean perimeter of the polygon and $\alpha$ denotes the sum of the Euclidean lengths of the arcs $A_i$'s contained in $\cal P.$

First observe that if $A_1$ and $A_2$ are not contained in $\cal{P},$ then the inequality holds trivially, since $\alpha=0.$ Hence we only consider the cases where at least one of the arcs $A_i$'s is in $\cal{P}.$

If $\mathcal{P}$ has as vertices the two endpoints of an $s_i,$ then $\cal P$ contains $A_i$ and another arc $\tau$ with the same endpoints as $A_i.$ Since $\varepsilon<y(q_1)(e^{T_H/2H}-1),$ Proposition \ref{prop-vertalig} implies that $A_i$ is the shortest $2H/y$-arc joining its endpoints. Hence, $\ell=\alpha+l(\tau)>2\alpha$ and the inequality follows.

In any other case, $\mathcal{P}$ contains at least an endpoint of each $s_i,$ therefore the perimeter of $\mathcal{P}$ inside $U$ is $(\ell-\alpha)>2d.$ Since $\alpha$ is at most two times $d/2,$ it holds that $\alpha\leq d<\ell-\alpha$ implying that $2\alpha<\ell+2H\mathcal{I}(\mathcal{P}).$

\end{proof}

\subsection{Admissible domains with $\{A_i\}=\emptyset$}

In the next proposition we not only exhibit domains for which Proposition \ref{prop-Aempty} implies the existence of a solution to the Dirichlet Problem attaining $-\infty$ on the arcs of type $B_i,$ but also we prove the existence of such domains around any point in a $2H/y-$curve. We will use these solutions as barriers to prove Theorem \ref{teo-Aempty}.

\begin{proposition}\label{prop-constBC}
Let $\gamma$ be a $2H/y-$curve. Given any point $p\in \gamma,$ there is a domain $U,$ containing $p$, such that
\begin{enumerate}
\item[i)] the boundary of $U$ consists of two arcs of circles $C_1$ and $C_2$ of $\kappa_{euc}(C_i)>2H/y$ and two arcs $B_1$ and $B_2$ of Euclidean curvature $-2H/y,$ if oriented inwards;
\item[ii)] the intersection $\partial U \cap \gamma$ is contained in $B_1\cup B_2;$ 
\item[iii)] $U$ admits a solution to the Dirichlet problem (Definition \ref{def_diriinfinito}).
\end{enumerate}
\end{proposition}

\begin{remark}
The main difference from the construction of the domain in Proposition \ref{prop-constBC} to the previous construction in Proposition  \ref{prop-consthoro} is that instead of replacing vertical segments, we will replace horizontal segments by curves of Euclidean curvature $-2H/y.$ The reason is because the inequality $2\beta<\ell-2H\mathcal{I}(\mathcal{P})$ has a minus sign and therefore is more delicate to be satisfied by any admissible polygon. If we did the procedure with vertical lines, non convex admissible polygons  with two vertices could arise and controling their length minus $\mathcal{I}\left(\mathcal{P}\right)$ is very complicated.
\end{remark}

\begin{proof}
We make a construction analogous to Proposition \ref{prop-consthoro}. Assuming that the tangent vector to $\gamma$ at $p=(x(p),y(p))$ is neither horizontal nor vertical, we take the rectangle $R$ as in the proof of Proposition \ref{prop-consthoro}. But here instead of being centered at $p,$ $R$ just contains $p$ and is such that $\gamma\cap R$ is a graph on the horizontal and vertical directions. Besides, we assume that $\gamma\cap \partial R$ consists in the two diagonal endpoints $q_1$ and $q_2.$ (See Figure \ref{fig-prop7}).

\begin{figure}[h]
\centering
\includegraphics[width=3.5cm]{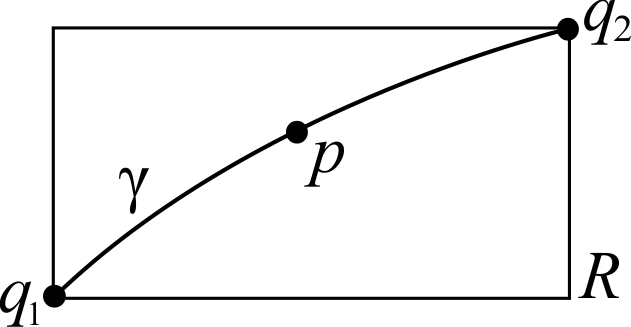}
\caption{Rectangle $R$.}
\label{fig-prop7}
\end{figure}

Denoting by $d$ the length of the diagonal of $R,$ we assume
\begin{equation}\label{eq-desigdB}
d<\min\left\{\frac{2y(p)}{3}(1-e^{-1/2H}), \frac{2y(p)}{8H+3}\right\}.
\end{equation}

As in the proof of last proposition, let $s_1$ and $s_2$ be two horizontal segments whose medium points are $q_1$ and $q_2,$ respectively, both of length $\varepsilon<2y(p)e^{-1/H}L_H,$ where $L_H$ is given by \eqref{eq-L}. We replace $s_i$ by $B_i,$ the shortest arc of $\kappa_{euc}=-2H/y$ (normal pointing  away from the segments $s_i$'s) that joins the endpoints of $s_i,$ with length less than $d/2$ (notice that we can ask the length of the arcs $B_i$'s to be as small as we want because their lengths become smaller with $\varepsilon$). Let $C_1$ (resp. $C_2$) be the semicircle that joins the left (resp. the right) endpoints of the segments $s_1$ and $s_2$, of diameter $d.$  We claim that the region $U$ bounded by $B_1,$ $C_1,$ $B_2$ and $C_2$ is the region that we are looking for.

\begin{figure}[h]
\centering
\includegraphics[width=6cm]{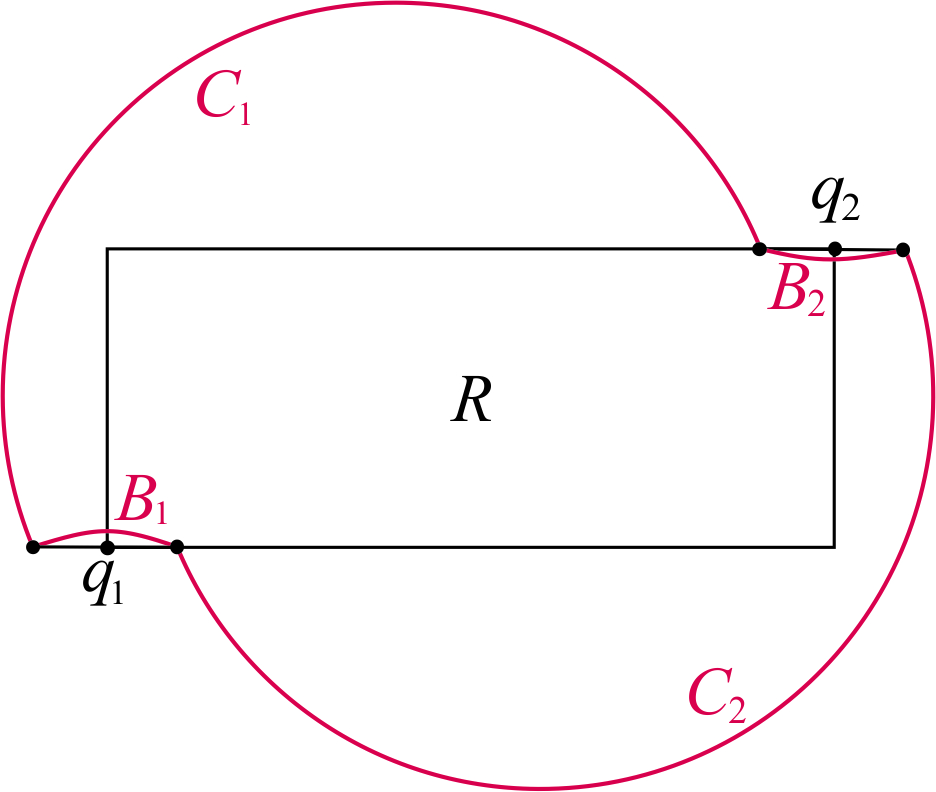}
\caption{Construction of the domain $U$ with empty $\{A_i\}$.}
\end{figure}

It follows from $d<\frac{2y(p)}{8H+3}<\frac{2y(p)}{2H+3}$, as in the proof of Proposition \ref{prop-consthoro}, that $\kappa_{euc}(C_i)>2H/y.$

Since $\varepsilon<2y(p)e^{-1/H}L_H$ and \eqref{eq-desigdB} implies that $y(q_1)> y(p)-d>y(p)e^{-1/2H},$ it follows from Proposition \ref{prop-horizalig} that $B_1$ and $B_2$ are the only arcs of $\kappa_{euc}=-2H/y$ joining their endpoints contained in $\bar{U}$ and, moreover, we have the existence of the arcs $B_i^*.$

Notice that in order to see there is no other arc joining the endpoints of $s_2$ (the upper segment) contained in $U,$ it is sufficient to check that $y_{min}\geq y(q_2)e^{-1/2H}.$  Since $y_{min}=\min\{y ; (x,y)\in \bar{U}\},$ it holds $y_{min}>y(p)-d-d/2= y(p)-3d/2,$  and then $y(q_2)>y(p)$ and  \eqref{eq-desigdB} implies $y_{min}\geq y(q_2)e^{-1/2H}.$

\vspace{0.5cm}

{\it Claim}: $U$ admits a solution to the Dirichlet problem (Definition \ref{def_diriinfinito}).

Let $\mathcal{P}$ be an admissible polygon in $\bar{U}$ and let $\ell$ denote its Euclidean perimeter. Because of Proposition \ref{prop-Aempty}, it is sufficient to prove the inequality:
\begin{equation}\label{eq-ineq}
\beta+2H\mathcal{I}(\mathcal{P})<(\ell-\beta).
\end{equation}

For any polygon we have $$\beta\leq \ell(B_1)+\ell(B_2)\leq d/2+d/2=d.$$
If $A(U)$ denotes the Euclidean area of $U,$ we have
$A(U)\leq \pi(d/2)^2+\varepsilon d\leq 2d^2$
and then
$$2H\mathcal{I}(\mathcal{P})\leq \frac{2H}{y_{min}}A(U)\leq
\frac{4Hd^2}{y_{min}} \leq   \frac{4Hd^2}{y(p)-3d/2}.$$

If $\mathcal{P}$ has 3 or 4 vertices, then ($\ell-\beta$) is at least $2d.$ If $\mathcal{P}$ has two vertices, each of them must belong to a different $B_i,$ and then ($\ell-\beta$)  is at least $2d$ as well.

Hence the inequality \eqref{eq-ineq} is a consequence of

$$d+ \frac{4Hd^2}{y(p)-3d/2}<2d,$$
which holds because of \eqref{eq-desigdB}.

If $p$ is a point of horizontal or vertical tangency in $\gamma,$ the construction becomes simpler but it is very similar. 
If $\gamma'(p)$ is vertical, we take a very thin rectangle centered at $p.$ We assume its height $d$ satisfies
$$d<\mbox{min}\left\{y(p)(1-e^{-1/2H}),\frac{2y(p)}{8H+1}\right\}.$$
We call $s_1$ and $s_2$ the bottom and top of the rectangle and we assume their length is $\varepsilon<2y(p)e^{-1/H}L_H$. Exactly as in the previous case, we replace $s_i$ by $B_i,$ the shortest arc of $\kappa_{euc}=-2H/y$ (normal pointing inwards) that joins the endpoints of $s_i,$ with length less than $d/2$; and we take $C_1$ (resp. $C_2$) the semicircle of diameter $d$ that joins the left (resp. the right) endpoints of the segments $s_1$ and $s_2$. Now the proof follows the exact same steps as in the previous case.
 
Let us assume that the tangent vector $\gamma'(p)$ is horizontal. Then let $R$ be a rectangle centered at $p$ with sides parallel to the axes such that:
\begin{enumerate}
\item[-] $\gamma$ intersects $\partial R$ in the vertical sides of $R,$ which we assume to have length 
\begin{equation}
h<2y(p)\frac{e^{-T_H/2H}-1}{1+e^{-T_H/2H}}=2y(p)\tanh(T_H/4H).
\label{eq-h}
\end{equation}

\item[-] The basis has length
\begin{equation}\label{eq-domadm}
b\leq \frac{y(p)}{4H+1}.\end{equation}
\end{enumerate}

We replace the vertical sides of $R$ by arcs $B_1$ and $B_2$  of curvature $-2H/y$ if oriented inwards. $B_1$ and $B_2$ have the same endpoints as the vertical sides and exist because of Proposition \ref{prop-vertalig}. Notice that we can apply Proposition \ref{prop-vertalig} because given any point $q$ in the basis of the rectangle, the hypothesis (\ref{eq-h}) implies that $h<y(q)(e^{-T_H/2H}-1)$, since $y(p)=y(q)-h/2.$

We also assume $h$ is small enough so that the length of $B_i$ is less than $b/2.$  We replace the horizontal segments of $\partial R$ by arcs of circle $C_1$ and $C_2$ such that the segments are diameters of the arcs. The domain $U$ is then bounded by $C_1,$ $B_1,$ $C_2$ and $B_2.$ We assume $U$ is small enough to admit a subsolution. Observe that $U^*$ is obtained by reflecting $B_i$ about the vertical line that joins its endpoints and that $U$ does not contain a $2H/y-$arc joining two vertically aligned vertices of $\partial U.$

As before, the assumption \eqref{eq-domadm} implies that the arcs $C_i$ have Euclidean curvature greater than $2H/y.$ 

To see that $U$ admits a solution to the Dirichlet problem, we analyse all possible admissible polygons $\mathcal{P}\subset \overline{U}.$

First notice that for any $\mathcal{P},$ it cannot have as vertices only the endpoints of an arc $B_i$ because, from Proposition \ref{prop-vertalig}, there is no other $2H/y-$curve connecting these points contained in $U.$
Therefore ($\ell-\beta$) is always greater than $2b.$ Besides, since the Euclidean area of $U$ satisfies $A(U)\leq bh+b^2\pi/4\leq 2b^2,$ as for the general case, we have

$$2H\mathcal{I}(\mathcal{P})\leq \frac{2HA(U)}{y_{min}}
\leq \frac{2H}{y(p)-b}2b^2\leq \frac{4Hb^2}{y(p)-b}.$$

Since $\beta\leq \ell(B_1)+\ell(B_2)<b,$ inequality \eqref{eq-desigdB} is a consequence of
$$b+\frac{4H}{y(p)-b}b^2<2b,$$
which holds from \eqref{eq-domadm}.

\end{proof}

\section{Existence Theorems - part II}

Let us finish by proving the remaining existence results.

\begin{proof}[Proof of Theorem \ref{teo-Aempty}:]
We proceed as in Proposition \ref{prop-Aempty} and consider the sequence of solutions $u_n$ to the Dirichlet Problem with boundary values $-n$ on $B_i^*$ and $\max\{f_i,-n\}$ on $C_i.$ Then Theorem \ref{theo-monotone} implies that the limit function $u$ exists and is a solution to the Dirichlet Problem in the convergence set $U.$

It remains to prove that the divergence set $V$ satisfies $V \cap \Omega=\emptyset.$ Notice that in this case, if $\kappa_{euc}(C_i)\equiv 2H/y,$ Lemma \ref{lem-localbarr} does not imply that $\left\{u_n\right\}$ is bounded from below in a neighborhood of $C_i.$ From Lemma \ref{lem-Vboundary}, a neighborhood of $C_i$ is either contained in $U$ or in $V.$ 

If it is contained in $U,$ by taking the neighborhood $U^-$ from Proposition \ref{prop-constBC}, bounded by 
$B'_1\cup C'_1\cup B'_2 \cup C'_2,$ with $C'_1\subset \Omega$ and $C'_2\cap \Omega=\emptyset$ we construct a barrier. Since $\left\{u_n\right\}$ is uniformly bounded in compact subsets of $U,$ $\inf_{C'_1} u_n$ is finite. Let $$m=\min\{\inf_{C'_1} u_n, \inf_{U^-\cap\partial\Omega} f_i\}.$$
Let $u^-: U^- \to \mathbb{R}$ be the solution of the Dirichlet Problem in $U^-$ taking value $m$ on $C'_1\cup C'_2,$ then the Maximum Principle implies that $u_n\geq u^-$ for all $n$ and the sequence is uniformly bounded. Proceeding as in the proof of Theorem \ref{theo_existence}, we conclude that $u=f_i$ on $C_i.$

If the neighborhood of $C_i$ is contained in $V,$ let $\mathcal{P}$ be an admissible polygon which is a connected component of $V\cap \Omega$ with $C_i$ on its boundary. From the definition of the flux, 
$$2H\mathcal{I}(\mathcal{P})=F_{u_n}(\Sigma B_i)+F_{u_n}(\Sigma C_i)
+F_{u_n}(\Sigma D_i),$$
for the arcs  $B_i, C_i$ of $\partial \Omega$ that are in $\mathcal{P},$ and $D_i$ the arcs of $\mathcal P$ contained in the interior of $\Omega$. Lemma \ref{lem-flux4} applied to the arcs $D_i$ and Lemma \ref{lem-flux5} to the arcs $C_i$ imply

$$\lim_{n\rightarrow+\infty} F_{u_n}(\Sigma D_i)=\ell-\beta-\Sigma \ell_{euc}(C_i),$$

$$\lim_{n\rightarrow+\infty} F_{u_n}(\Sigma C_i)=\Sigma \ell_{euc}(C_i).$$
Besides, from Lemma \ref{lem-flux1}, $|F_{u_n}(\Sigma B_i)|\leq \beta.$
Then,
$$2H\mathcal{I}(\mathcal{P})\geq -\beta+ \Sigma \ell_{euc}(C_i)+(\ell-\beta-\Sigma \ell_{euc}(C_i))=\ell-2\beta,$$
contradicting inequality \eqref{eq-ineqbeta}. Thus, $\left\{u_n\right\}$ is uniformly bounded in a neighborhood of $C_i$ and the proof may follow the steps of Proposition \ref{prop-Aempty}. Therefore, the set $V\cap \Omega$ must be empty and the barrier argument from Theorem \ref{theo_existence} implies that $u$ takes the boundary values $f_i$ on $C_i.$

It remains to show that $u$ diverges to $-\infty$ on the arcs $B_i.$ Since $\left\{u_n\right\}$ is a monotonically decreasing sequence $u_n\leq u_0\leq M,$ it is bounded from above in $\Omega.$ As in Theorem \ref{teo-Bempty}, given $p\in B_i,$ let $U$ be a neighborhood of $p$ for which Theorem \ref{theo_existence} (existence result) applies. Take $v_k$ the solution to the Dirichlet problem in $U\cap \Omega$ with boundary data $M$ on $\partial U \cap \Omega$ and $-k$ on $\partial \Omega \cap U.$ Then for $n\geq k,$ $u_n\leq v_k$ by the Maximum Principle. Therefore, $u\leq v_k$ for all $k$ and $u$ diverges to $-\infty$ on $U\cap \partial \Omega.$

\end{proof}

\begin{theorem}\label{teo-ABC}
Let $\Omega$ be an admissible domain such that the family $\{C_i\}$ is non empty. Assume that $\kappa_{euc}(C_i)\geq 2H/y.$ Then the Dirichlet Problem has a solution in $\Omega$ if and only if
\begin{equation}\label{eq-ineqalphaEbeta1}
2\alpha <\ell+2H\mathcal{I}(\mathcal{P})
\text{ and }
2\beta <\ell-2H\mathcal{I}(\mathcal{P})
\end{equation}
for all admissible $\mathcal{P}.$
\end{theorem}

\begin{proof}
Let $\Omega^*$ be the  correspondent domain to $\Omega.$
Let $u_n$ be the solution in $\Omega^*$ to the Dirichlet problem with boundary data $u_n=n$ on the arcs of the family $\{A_i\},$ $u_n=-n$ on the ones from $\{B_i^*\}$ and $u_n=f_n$ on the arcs of the family $\{C_i\}.$ The function $f_n$ coincides with $f_i$ if $|f_i|< n,$  takes value $n$ if $f_i>n$ and $-n$ otherwise.

Define also $u^+:\Omega\rightarrow\mathbb{R}$ as the solution that goes to $+\infty$ on $A_i,$ vanishes on $B_i$ and coincides with $\max\{f_i,0\}$ on $C_i.$ It exists from Theorem \ref{teo-Bempty}. By the Maximum Principle, $u_n\leq u^+$ for all $n.$ Analogously, there is $u^-:\Omega\rightarrow\mathbb{R}$ that vanishes on $A_i$ and diverges to $-\infty$ on $B_i$ which, from the Maximum Principle, is below all $u_n.$

Therefore, the sequence $\left\{u_n\right\}$ is uniformly bounded in compact subsets of $\Omega$ and the already presented arguments lead us to conclude that the solution exists.

It remains to see that $u$ takes the required boundary data. For the boundary arcs $A_i,$ the construction of the function $v_k$ from Theorem \ref{teo-Bempty} with $u_0$ replaced by $u^-,$ implies that $u$ goes to infinity on $A_i.$ Also for the arcs $B_i,$ the argument from Theorem \ref{teo-Aempty} with $u_0$ replaced by $u^+$ gives the result. If $\kappa_{euc}(C_i)\ge2H/y,$ with strict inequality except for isolated points, then the sequence $(u_n)$ is uniformly bounded in a neighborhood of $C_i$ by Lemma \ref{lem-viz}. If $\kappa_{euc}(C_i)=2H/y,$ Lemma \ref{lem-viz} only provides the boundedness from above. To obtain the boundedness from below, we proceed as in the proof of Theorem \ref{teo-Aempty} using Proposition \ref{prop-constBC}.

The necessity of conditions \eqref{eq-ineqalphaEbeta1} follows as in the previous results.

\end{proof}

Theorem \ref{teo-Cempty} was stated in the Introduction and here we describe the steps of its proof.%\pati{sugiro remover a frase que vem antes daqui.} {\color{green}nao entendi. Qual frase?}\pati{esse paragrafo. mas acho que faz sentido mantelo, entao esquece.}

\begin{proof}[Proof of Theorem \ref{teo-Cempty}]
Differently from the last result, if one takes a sequence $u_n$ of solutions, there is no guarantee that it is bounded in compact subsets of $\Omega.$ That is why this proof is more delicate. Nevertheless the idea used in \cite{JS} to demonstrate the existence result for surfaces in $\mathbb{R}^3$ also works in this setting. We briefly describe the construction of two functions that will play the roles of $u^+$ and $u^-.$

For each $i$ such that $A_i$ is an arc of $\partial\Omega,$ let $u_i^+$ be defined in $\Omega^*$ as the solution to the Dirichlet problem with boundary values $u_i^+=+\infty$ on $A_i$ and zero on the remaining part of the boundary.
The existence of $u_i^+$ follows from Theorem \ref{teo-Bempty}. Besides, since $\Omega$ is an admissible domain, it has a subsolution and therefore, the Maximum Principle implies the existence of $N>0$ such that $u^+>-N.$

If $i$ is such that $B_i$ is an arc of $\partial\Omega,$ let $u_i^-$ be defined in $\Omega_i,$ the domain bounded by $\left(\cup_j A_j\right) \cup B_i^* \cup \left(\cup _{j\neq i} B_j\right)$ as the solution to the Dirichlet problem with boundary values $u_i^-=-\infty$ on $\cup _{j\neq i} B_j$ and zero on the remaining part of the boundary.
The existence of $u_i^-$ follows from Theorem \ref{teo-Aempty}. For that we observe that any admissible polygon $\mathcal{P}$ in $\overline{\Omega_i}$ satisfies \eqref{eq-ineqbeta}. If $\mathcal{P}$ is contained in $\Omega,$ this is a consequence of the hypothesis \eqref{eq-eqAB}. If $\mathcal{P}$ contains the domain $W$ bounded by $B_i$ and $B_i^*,$ then $\mathcal{P}=\mathcal{P}_1\cup W$ for $\mathcal{P}_1$ the admissible polygon in $\overline{\Omega}$ obtained removing $W$ from $\mathcal{P}.$ Therefore \eqref{eq-ineqbeta} implies that $2\beta(\mathcal{P}_1)<\ell(\mathcal{P}_1)-2H\mathcal{I}(\mathcal{P}_1).$

Since from Lemma \ref{lem-flux2}, $\ell(\partial W)\geq 2H\mathcal{I}(\mathcal{P}_1),$ we get
$$2\beta(\mathcal{P}_1)<\ell(\mathcal{P}_1)-2H\mathcal{I}(\mathcal{P}_1)+\ell(\partial W)- 2H\mathcal{I}(W).$$

Observing that $\beta(\mathcal{P}_1)=\beta(\mathcal{P})+\beta_i$ and 
$\ell(\mathcal{P}_1)=\ell(\mathcal{P})+\beta_i-\beta_i^*,$ it follows

$$2\beta<\ell(\mathcal{P})-2H\mathcal{I}(\mathcal{P}).$$

Set $u^+:\Omega^*\rightarrow\mathbb{R}$ as 
$u^+=\max_{i} u^+_i$
and $u^-:\Omega\rightarrow\mathbb{R}$ as 
$u^-=\min_{i} u^-_i.$

Let $v_n$ be the solution to the Dirichlet problem in $\Omega^*$ with boundary values $n$ on $A_i$ and $0$ on $B_i^*.$ For each $n>1,$ for $c\in (0,n),$ define the sets
$$E_c^n=\{v_n-v_0>c\} \text{  and  }
F_c^n=\{v_n-v_0<c\}.$$

For $c$ sufficiently close to $n,$ $E_n^c$ is a union of distinct and disjoint subsets of $\Omega^*,$ each of them containing one arc $A_i.$ We suppress the dependence on $n$ and denote each component of $E_n^c$ by $E_i^c.$
Define $\mu(n)$ as the infimum of the constant $c<n$ such that $E_i^c$ are all distinct and disjoint. Then define $u_n=v_n-\mu(n).$

We claim that for $M=N+\sup_{\Omega^*}v_0,$ it holds $u^--M\leq u_n$ in $\Omega$ and $u_n\leq u^++M$ in $\Omega^*.$ This claim follows from the Maximum Principle as it did in other ambient spaces, see for instance \cite{S}.

Consequently, $\left\{u_n\right\}$ is uniformly bounded on compact subsets of $\Omega$ and, therefore, it has a converging subsequence. Define $u$ as the limit of this sequence. Once again following the ideas of \cite{S}, one proves that $\mu(n)$ and $n-\mu(n)$ diverge to $+\infty.$ Since $u_n=-\mu(n)$ on $B_i^*$ and $u_n=n-\mu(n)$ on $A_i,$ the conclusions in the proofs of Theorems \ref{teo-Bempty}  and \ref{teo-Aempty} will prove that $u$ is the solution to the Dirichlet problem.

Conversely, the necessity of conditions \eqref{eq-eqAB} and \eqref{eq-ineqalphaEbeta} follows as in Theorem \ref{teo-Bempty} and Proposition \ref{prop-Aempty}.
\end{proof}

\section{Admissible domains with empty $\{C_i\}$} \label{sec-exemplo}

We devote this section to prove the existence of domains for which Theorem \ref{teo-Cempty} applies. We remark that even in $\mathbb{R}^2$ where we have a 6-dimensional isometry group, the construction of these domains requires some effort, as it was presented in \cite{S}.
%The Section is divided in the following steps \pati{(para apagar depois)}
%
%-describe the construction of $\Omega_s,$ $s\in [0,s_0].$ok
%
%-- find $D+,$ $D-,$ $E+,$ $E-,$ depending on $s.$ok
%
%-- prove the existence of $s_0$ such that $\Omega_s$ is connected for $s\in [0,s_0).$ ok
%
%- notice that $\Omega_0$ satisfies one inequality and $\Omega_{s_0}$ the other. \pati{tem buracos aqui}
%
%-- conclude the existence of $s^\star\in [0,s_0]$ for which equality holds. ok
%
%- prove that $\Omega^\star=\Omega_{s^\star}$ satisfies hypothesis of Theorem 3. ok

%\subsection{Construction of $\Omega_s$}

In order to search for domains to apply Theorem \ref{teo-Cempty} we start with the domain $U$ bounded by one loop of a $2H/y-$curve $\gamma,$ which starts at a $P_3-$point, goes down to the $P_1=(0,y_0)-$point and then comes back up to $P_3=(0, y_0e^{\frac{1+T_H}{2H}}).$ From \eqref{eq-curva-x=x(y)}, we may assume $\gamma$ is parameterized by 
$\gamma^+(t)=(x(t),y_0e^{t/2H})$ and $\gamma^-(t)=(-x(t),y_0e^{t/2H}),$ $t\in [0,1+T_H],$ $x(t)$ given by \eqref{eq-curva-x=x(y)}.

%\begin{equation}\label{eq-curva-x=x(y)}
%x(t)=\frac{y_0e^{1/2H}}{2H}\int_{-1}^{t-1} \frac{-ue^{u/2H}}{\sqrt{1-u^2}}du= \frac{y_0e^{1/2H}}{2H}\int_{-1}^{t-1} g_H(u)du
%\end{equation}

We proceed to construct a family of domains $\Omega_s\subset U.$ Given $s\in [0,1],$ let 
$D^+(s)=\gamma^+(s)$ and $D^-(s)=\gamma^-(s).$ Let $E^+(s)$ be the second intersection point of the vertical line through $D^+(s)$ with $\gamma^+;$ and $E^-(s)$ be the fourth vertex of a rectangle with sides parallel to the axes and vertices $D^+(s),$ $D^-(s)$ and $E^+(s)$. (See Figure \ref{emptyC}).

From the definition of $\gamma,$ one can see that $E^+(s)=(x(s),y_0e^{\varphi(s)/2H}),$ where $\varphi:[0,1]\to [1,1+T_H]$ is a bijection given by \begin{equation}\label{eq-defvarphi}\int_{s-1}^{\varphi(s)-1} g_H(u)du=0.\end{equation}

The domain $\Omega_s$ is bounded by the subarc of $\gamma^+$ from $D^+(s)$ to $E^+(s)$ and of $\gamma^-$ from $D^-(s)$ to $E^-(s),$ which we name $A^+$ and $A^-,$ respectively. The remaining part of the boundary of $\Omega_s$ consists in two arcs $B_D$ and $B_E$. $B_D$ is the unique $2H/y-$arc (oriented downwards) that joins $D^-(s)$ and $D^+(s)$ and $B_E$ is the unique $2H/y-$arc (oriented upwards) that joins $E^-(s)$ and $E^+(s)$ and it is not contained in $\gamma.$ (See Figure \ref{emptyC})

\begin{figure}[h!]
\centering
\includegraphics[height=4cm]{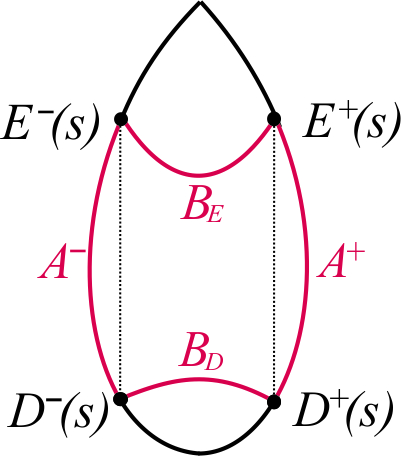}
\caption{Construction of $\Omega_s$.}
\label{emptyC}
\end{figure}

For $s=0,$ the $B-$type arcs are empty and $\Omega_0=U.$ For small values of $s$, $\Omega_s$ is a well defined domain and then there is $s_0$ for which $B_D(s_0)$ intercepts $B_E(s_0)$ at one point and for $s>s_0,$ $\Omega_s$ is no longer a domain.

%\begin{figure}[h!]
%\centering
%\includegraphics[scale=.4]{U_Hmeio.png}
%\caption{A domain $\Omega_s$ for $H=1/2.$ The arcs $B_D$ and $B_E$ are just illustrations.}
%\label{fig-omegas}
%\end{figure}

Let $E$ be the $P_1-$point of the arc $B_E$ and let $(0, y_0e^{e(s)/2H})$ be its coordinates. Let also $D$ be the $P_4-$ point of $B_D,$ $D=(0, y_0e^{d(s)/2H}).$ Some computations imply that for $s\in (0,1)$ the functions $e$ and $d$ are given by the following expressions:

\begin{equation}\label{eq-function_e}
e^{e(s)/2H}\int_{-1}^{-1+\varphi(s)-e(s)}g_H(u)du=\int_{-1}^{-1+s}g_H(u)du
\end{equation}
and 
\begin{equation}\label{eq-function_d}
e^{(d(s)-2)/2H}\int_{1-(d(s)-s)}^{1}-g_H(u)du=\int_{-1}^{-1+s}g_H(u)du.
\end{equation}

%\includegraphics[scale=.5]{explicacao.png}

%\ana{Para mim a defini��o da $\varphi$ ja estava bem posta sobre todo o intervalo $[0,1]$, sem precisar pensar em extensao. Mas sobre o $s$, eu ainda acho que soh faz sentido para $s\neq 0,1$ pois os arcos B sao vazios. Como a gente nao usa em nenhum lugar a existencia das funcoes $d(s)$ e $e(s)$ para $s=0$ , acho que seria melhor nao colocar nada sobre isso, e escrever a claim pra $0<s\leq s_0$. Ah, e acho que a figura acima poderia ser retirada}\pati{Com a figura, estendendo continuamente $\varphi$ para os extremos do intervalo $(0,1),$ temos $\varphi(1)=1$ e $\varphi(0)=1+T_H.$ De fato, a express\~ao \ref{eq-function_e} s\'o \'e uma definição boa para $s\in (0,1),$ mas olhando a figura e conhecendo a proposição \ref{prop-horizalig} (horiz align), temos que $e(0)=\varphi(0).$}

For $s\in [0,s_0],$ define $\alpha(s)$ and $\beta(s)$ as the length of the $A-$type arcs and $B-$type arcs, respectively, that bound $\Omega_s.$
% Analogously define $\beta(s)$ and $\mathcal{I}(s)=\mathcal{I}(\Omega_s).$ 

As a first step to find a domain for Theorem \ref{teo-Cempty} we have the following  fact.

\begin{theorem}
There is $s^\star\in (0,s_0)$ such that $\alpha(s^\star)=\beta(s^\star)+2H\mathcal{I}(\Omega_{s^\star}).$
\end{theorem}

\begin{proof}
Since $F(s):= \alpha(s)-\beta(s)-2H\mathcal{I}(\Omega_{s})$ is a continuous function on $s\in [0,s_0],$ it is sufficient to see that $F(0)>0$ and $F(s_0)<0.$

For $s=0$, we are in the same situation as in Lemma \ref{lem-vertalign}, which asserts that $F(0)>0.$ For $s=s_0,$ some computations must be done. First notice that $s_0$ is the solution of $d(s)=e(s).$ We denote $d(s_0)$ by $d_0$ and $\varphi(s_0)$ by $\varphi_0.$

Besides, as it was already computed in the proof of Lemma \ref{lemma-horizdesig}, the length of an arc in a $2H/y-$curve with $P_1=(z,0)-$point and that goes from height $ze^{t_1/2H}$ to height $ze^{t_2/2H}$ is $$L(t_1,t_2)=\frac{ze^{1/2H}}{2H}\int_{t_1}^{t_2}\frac{g_H(u)}{-u}du.$$

Therefore $\alpha(s_0)<\beta(s_0)$ is equivalent to

$$\int_{s_0-1}^{\varphi_0-1}\frac{g_H(u)}{-u}du<e^{\frac{(d_0-2)}{2H}}\int_{1-(d_0-s_0)}^1\frac{g_H(u)}{-u}du + e^{\frac{d_0}{2H}}\int
_{-1}^{-1+\varphi_0-d_0}\frac{g_H(u)}{-u}du$$
%or to \pati{para apagar a proxima linha}
%$$\int_{s_0-1}^{\varphi_0-1}\frac{e^{u/2H}}{\sqrt{1-u^2}}du<\int
%_{s_0-1}^{d_0-1}\frac{e^{v/2H}}{\sqrt{1-(d_0-2-v)^2}}dv+\int
%_{d_0-1}^{\varphi_0-1}\frac{e^{v/2H}}{\sqrt{1-(d_0-v)^2}}dv$$
or, after a change of variables, to
$$\int_{s_0-1}^{\varphi_0-1}\frac{e^{u/2H}}{\sqrt{1-u^2}}du<\int
_{s_0-1}^{\varphi_0-1}\frac{e^{u/2H}}{\sqrt{1-(1-|u-c|)^2}}du$$
for $c=d_0-1.$

A straightforward computation splitting the interval  $(s_0-1,\varphi_0-1)$ as   $\left(s_0-1,c\right]\cup \left[c,\varphi_0-1\right)$ shows that the above inequality for the integrands holds on  $\left(s_0-1,c\right]$ if $d_0<2s_0$ and holds on  $\left[c,\varphi_0-1\right)$ if $2(\varphi_0-1)<d_0$.

Hence it is enough to show that $2(\varphi_0-1)<d_0<2s_0.$

\vspace{.3cm}

%{\bf Claim 1:} The function $e$ satisfies $e(s_0)\geq 2(\varphi(s_0)-1).$ 

%\pati{a gente prova para qualquer $s$ com a mesma prova do $s_0$ fixo. Vou escrever no qquer e depois vemos o que fica melhor.}

{\it Claim 1:} The function $e$ satisfies $e(s)\geq 2(\varphi(s)-1)$ for all $s\in (0,1).$

Fix $s$. Consider $Q_s=(0,y_0e^{2 (\varphi(s)-1)/2H})$ and take  $\gamma_{Q_s}$ the $2H/y-$curve with $Q_{s}$ as its $P_1-$point. The intersections of $\gamma_{Q_s}$ and the horizontal line $\{y=y_0e^{\varphi(s)/2H}\}$ are at a distance $2x_{Q_s}$ % greater than $2x(s)$ because this distance is
given by
$$2{x_{Q_s}}=\frac{y_0e^{1/2H}}{2H}e^{\frac{2(\varphi(s)-1)}{2H}}\int_{-1}^{1-\varphi(s)}g_H(u)du,$$ which is greater than $2x(s)$ if and only if% and therefore $2x(s)>2x_{Q_s}$ becomes equivalent to

%\pati{para apagar a primeira linha:}
%$$\frac{y_0e^{1/2H}}{2H}e^{\frac{2(\varphi(s)-1)}{2H}}\int_{-1}^{1-\varphi(s)}g_H(u)du,$$
$$\int_{-1}^{s-1}g_H(u)du<e^{\frac{2(\varphi(s)-1)}{2H}}\int_{-1}^{1-\varphi(s)}g_H(u)du,$$ which is a consequence of the positivity of the integrand in the interval $(-1, 1-\varphi(s))$ that contains $(-1, s-1)$ (from the definition of $\varphi$ \eqref{eq-defvarphi}) and of the fact that $e^{\frac{2(\varphi(s)-1)}{2H}}>1.$

Consider the function $l$ that associates to each point $P$ on the segment $V=\{(0,y); y\in (y_0e^{(\varphi(s)-1-T_H)/2H}, y_0e^{\varphi(s)/2H})\}$ the distance between the two intersections of $\gamma_{P}$ with $\{y=y_0e^{\varphi(s)/2H}\}.$ As described in the proof of Proposition \ref{prop-horizalig} (see \eqref{eq-defl}), $l$ is continuous, takes value 0 on the extrema of the segment $V$ and has a unique critical point which is a maximum. Since $l(E)= l((0,y_0))$ and $l(Q_s)> l(E),$ we have that the point $Q_s$ is necessarily between $E$ and $(0,y_0)$. In particular, $Q_s$ is below $E$ and the claim follows. 

%\supseteq \{x=0, y\in (y_0, y_0e^{2(\varphi(s)-1)/2H})\}$ 

%Observe that the point $E=E(s)\in V_1$ is the unique other point in $V_1$ with $l$ assuming the same value as $(0,y_0).$ Now we look at the smaller segment $V_0$ and notice that since $l$ is equal to $l((0,y_0))$ in one extremal and greater than $l((0,y_0))$ on the other one, which is $Q_s.$ From the profile of the function $l,$ $E$ must be in $V_1\backslash V_0.$
 
%\begin{center}
%\begin{figure}
%\includegraphics[scale=.5]{function_d.png}
%\caption{Profile of $l$}
%\end{figure}
%\end{center}

Since $d_0=d(s_0)=e(s_0),$ the first inequality is proved.

\vspace{0.5cm}

{\it Claim 2:} The function $d$ satisfies $d(s)\leq 2s$ for all $s\in (0,s_0].$

For fixed $s\in (0,s_0],$ define $F_s(t)$ as the left-hand side of expression \eqref{eq-function_d}, that is, $$F_s(t)=e^{(t-2)/2H}\int_{1-(t-s)}^{1}-g_H(u)du,$$ and consider $X(s)=\int_{-1}^{-1+s}g_H(u)du$ (the right-hand side of \eqref{eq-function_d}). It is easy to see that $F_s(s)=0<X(s)$ and that $$F_s(2s)=e^{2(s-1)/2H}\int_{-1}^{-1+s}\frac{-ue^{-u/2H}}{\sqrt{1-u^2}}du>X(s),$$ once $s\in [0,1].$  Since $d$ is the solution of $F_s(d)=X(s)$ and $F_s$ is continuous, then $s<d(s)<2s.$ 

Therefore, the second inequality is true and the theorem is proved.

%\pati{(here it is why we exclude $s=0$ from the Claim, since $X(0)=0)$}

\end{proof}

We conclude by applying Theorem \ref{teo-Cempty} to $\Omega_{s^\star}.$

\begin{theorem}
The domain $\Omega_{s^\star}$ for $y_0$ small enough is an admissible domain with $\{C_i\}$ empty that admits a solution to the Dirichlet problem.
\end{theorem}

\begin{proof}
Let us denote $\Omega_{s^\star}$ by $\Omega.$ From the last result, 
\begin{equation}\label{eq-equality}
\alpha=\beta+2H\mathcal{I}\left(\Omega\right).
\end{equation}

We assume $y_0$ is small enough so that a subsolution exists in $\Omega.$

It remains to see that for any admissible polygon $\mathcal{P}$ properly contained in $\overline{\Omega},$ it holds that
$$2\alpha <\ell+2H\mathcal{I}(\mathcal{P}) \text{ and }
2\beta <\ell-2H\mathcal{I}(\mathcal{P}).$$

We analyse each possibility depending on the number of vertices of $\mathcal{P}:$
\begin{enumerate}
\item If $\mathcal{P}$ has two vertices, three configurations are possible:
\begin{enumerate}
\item[1.1] Two vertices horizontally aligned: Actually this configuration is not possible because $B_1$ and $B_2$ are the only $2H/y-$arcs in $\overline{\Omega}$ connecting these points;
\item[1.2] Two vertices vertically aligned:
The unique possibility inside $\overline{\Omega}$ is to have $\partial\mathcal{P}$ as an $A-$type arc (either $A^-$ or $A^+$) and its reflection about the vertical line that connects its endpoints. For this case, $\ell=2\alpha<\ell+2H\mathcal{I}(\mathcal{P})$ and $2\beta=0<\ell-2H\mathcal{I}(\mathcal{P})$ from Lemma \ref{lem-vertalign}.
\item[1.3] The vertices are the ends of a diagonal of $R;$ say $D^-$ and $E^+.$ The first inequality is trivial. Let us check the second one.

Let us denote by $\alpha^-$ and $\alpha^+$ the length of the arcs $A^-$ and $A^+,$ respectively. We know $\alpha^-=\alpha^+.$ And let us denote by $\beta_D$ and $\beta_E$ the length of the arcs $B_D$ and $B_E,$ respectively.

Let $\delta$ be a $2H/y-$arc from $D^-$ to $E^+.$ If the curvature vector of $\delta$ points to the left (notice it points in only one direction because since $\delta$ is contained in $\bar{\Omega}$ it cannot have neither a $P_1-$point nor a $P_4-$point), the Maximum Principle implies that $\delta$ is outside the domain $U^-$ bounded by $A^-$ and its reflection about the vertical line through its endpoints. Consider the curve obtained joining $\delta$ to the horizontal segment $h$ from $E^+$ to $E^-.$ This curve connects $D^-$ to $E^-$ and is outside the convex domain $U^-.$ Therefore its length is greater than the length of its projection on $U^-,$ which is exactly the reflection of $A^-.$ Therefore, $\ell(\delta)+\ell(h)>\alpha^-$ and then $\ell(\delta)+\beta_i>\alpha^-$ for both $i=D$ and $E.$ The same argument applies if $\delta$ has curvature vector pointing to the right, with $A^+$ replacing $A^-.$

Since $\beta=0$ and
$$2H\mathcal{I}(\mathcal{P})<2H\mathcal{I}(\Omega)=\alpha^++\alpha^--\beta_D-\beta_E<\delta_1+\delta_2=\ell,$$ the second inequality follows.

%By looking at the shape of the $2H/y-$curves, we see that there are three possibilities for $\delta:$
%
%-$\delta$ is a subarc of a $2H/y-$curve between its $P_1$ and $P_2^+.$
%
%-$\delta$ is a subarc of a $2H/y-$curve between its $P_1$ and $P_3$ points and contains $P_2^+.$
%
%-$\delta$ is a subarc of a $2H/y-$curve between its $P_1$ and $P_3$ points and contains $P_2^-.$
%
%
%
%
%Let $U_{2}$ be the domain bounded by $A_2$ and its reflection about the vertical line that connects its endpoints. For the first and third cases, notice that $\delta\cap U_{2}=-Q_1$ because if it were not so, $\delta$ would get in and out of $U_{2}$ and this would contradict the Tangency principle. Consequently the length of $\delta$ is greater than the length of its projetion on $U_{21},$ which is the reflection of $A_2.$ Hence $l(\delta)>\alpha_2=\alpha_1.$ For the second case, the same argument with $U_{2}$ replaced by $U_1$ (bounded by $A_1$ and its reflection) implies $l(\delta)>\alpha_1.$ Therefore in all cases $l(\delta)>\alpha_1.$
%
%The first inequality is then a straight consequence of $l(\delta)>\alpha_1$ and for the second we must verify that $2HI(P)<l(\delta_1)+l(\delta_2).$
%It holds that $2HI(P)<2HI(\Omega)=\alpha_1+\alpha_2-\beta_1-\beta_2<\alpha_1+\alpha_2\leq \delta_1+\delta_2.$
%
%Since reflections about vertical lines are isometries, the other diagonal possibility is equivalent to this one and we verified that any polygon with two vertices satisfies the two inequelities.
\end{enumerate}

\item If $\mathcal{P}$ has 3 vertices, two configurations are possible:
\begin{enumerate}
\item[2.1] $\mathcal{P}$ consists in one $A-$type arc, say $A^-,$ one $B-$type arc, say $B_E,$ and one interior arc $\delta$ as described above.

Here $\alpha=\alpha^-$ and, as shown above, $2\alpha<(\alpha^-+\beta_E+\delta)\leq \ell+2H\mathcal{I}(\mathcal{P}).$
Besides, $\beta=\beta_E$ and the inequality $2\beta<\ell-2H\mathcal{I}(\mathcal{P})$ is equivalent to
$\beta_E+2H\mathcal{I}(\mathcal{P})<\alpha^-+\delta,$ which from \eqref{eq-equality} is equivalent to
$$2\alpha^--\beta_D-2H\mathcal{I}(\Omega)+2H\mathcal{I}(\mathcal{P})<\alpha^-+\delta,$$ which is implied by the facts that
$\alpha^--\beta_D<\delta$ and $2H\mathcal{I}(\mathcal{P})<2H\mathcal{I}(\Omega).$ 

\item[2.2] $\mathcal{P}$ consists in one $B-$type arc ($B_E$ or $B_D$), and two interior arcs: one arc $\delta$ as described above and one arc $A',$ the reflection from either $A^-$ or $A^+$.
Here $\alpha=0$ and the first inequality is trivial.
The second one is again equivalent to $\beta_i+2H\mathcal{I}(\mathcal{P})<\alpha^-+\delta,$ which holds as above.
\end{enumerate}

\item If $\mathcal{P}$ has 4 vertices, two configurations are possible:
\begin{enumerate}
\item[3.1] $\mathcal{P}$ contains three boundary arcs and one arc $A',$ reflected from $A^-$ or $A^+.$
The first inequality is trivial.
For the second one we need that $\beta+2H\mathcal{I}(\mathcal{P})=\beta_D+\beta_E+2H\mathcal{I}(\mathcal{P})<2\alpha^-.$ But $\mathcal{I}(\mathcal{P}) <\mathcal{I}(\Omega)$ and \eqref{eq-equality} imply that $\beta_D+\beta_E+2H\mathcal{I}(\mathcal{P})<2\alpha^-.$

\item[3.2] $\mathcal{P}$ contains the two boundary arcs $B_D$ and $B_E$, and two $A'$s, reflected from $A^-$ and $A^+.$
Once again the first inequality is trivial and the second one follows as above.

\end{enumerate}
\end{enumerate}

\end{proof}

\begin{flushleft}
\textsc{Departamento de Matem\'atica}

\textsc{Universidade Federal de Santa Maria}

\textsc{Av. Roraima 1000, Santa Maria RS, 97105-900, Brazil}

\textit{Email:} patricia.klaser@ufsm.br
\end{flushleft}

\begin{flushleft}
\textsc{Mathematics Department}

 \textsc{Princeton University}

\textsc{Fine Hall, Washington Road, Princeton NJ, 08544, USA}

\textit{Email:} amenezes@math.princeton.edu
\end{flushleft}

\end{document}